\documentclass[reqno,10pt,centertags]{amsart} 
\usepackage{amsmath,amsthm,amscd,amssymb,latexsym,esint,upref,stmaryrd,enumerate,color,
verbatim,yfonts,multicol}
\usepackage{color,float}
\usepackage{hyperref}

\usepackage{tikz}
\usetikzlibrary{calc,decorations.markings}
\usepackage{pgfplots}
\pgfplotsset{compat=1.15}

\usepackage{graphicx}

\newcommand*{\mailto}[1]{\href{mailto:#1}{\nolinkurl{#1}}}
\newcommand{\arxiv}[1]{\href{http://arxiv.org/abs/#1}{arXiv:#1}}

\newcommand{\R}{{\mathbb R}}
\newcommand{\N}{{\mathbb N}}
\newcommand{\Z}{{\mathbb Z}}
\newcommand{\C}{{\mathbb C}}

\newcommand{\bbC}{{\mathbb{C}}}

\newcommand{\bbN}{{\mathbb{N}}}
\newcommand{\bbP}{{\mathbb{P}}}

\newcommand{\bbR}{{\mathbb{R}}}

\newcommand{\cA}{{\mathcal A}}
\newcommand{\cB}{{\mathcal B}}

\newcommand{\cF}{{\mathcal F}}

\newcommand{\cH}{{\mathcal H}}

\newcommand{\cS}{{\mathcal S}}

\newcommand{\beq}{\begin{align}}
\newcommand{\enq}{\end{align}}

\renewcommand{\a}{\alpha}
\renewcommand{\b}{\beta}
\newcommand{\g}{\gamma}

\DeclareMathOperator{\dom}{dom}

\DeclareMathOperator{\tr}{tr}

\DeclareMathOperator{\Ai}{Ai}
\DeclareMathOperator{\Bi}{Bi}

\DeclareMathOperator{\erf}{erf}

\newcommand{\Res}{\text{\rm Res}}
\renewcommand{\Re}{\text{\rm Re}}
\renewcommand{\Im}{\text{\rm Im}}
\renewcommand{\ln}{\text{\rm ln}}

\newcommand{\no}{\notag}
\newcommand{\lb}{\label}
\newcommand{\f}{\frac}

\newcommand{\ol}{\overline}

\newcommand{\wti}{\widetilde}
\newcommand{\Oh}{O}
\newcommand{\oh}{o}
 
\newcommand{\dott}{\,\cdot\,}

\renewcommand{\dot}{\overset{\textbf{\large.}}}
\renewcommand{\ddot}{\overset{\textbf{\large.\large.}}}

\newcommand{\bi}{\bibitem}

\renewcommand{\le}{\leqslant}

\let\geq\geqslant
\let\leq\leqslant

\newcommand{\al}{\a}
\newcommand{\be}{\b}

\newcommand{\Lr}{{L^2((a,b);rdx)}} 
\newcommand{\AC}{{AC([a,b])}}
\newcommand{\ACl}{{AC_{loc}((a,b))}}

\makeatletter
\def\theequation{\@arabic\c@equation}

\allowdisplaybreaks 
\numberwithin{equation}{section}

\newtheorem{theorem}{Theorem}[section]

\newtheorem{corollary}[theorem]{Corollary}

\newtheorem{hypothesis}[theorem]{Hypothesis}
\newtheorem{example}[theorem]{Example}

\theoremstyle{remark}
\newtheorem{remark}[theorem]{Remark}

\begin{document}

\title[Spectral $\zeta$-function for Regular Sturm--Liouville Operators]{Spectral $\boldsymbol{\zeta}$-Functions and $\boldsymbol{\zeta}$-Regularized Functional Determinants for Regular Sturm--Liouville Operators} 

\author[G.\ Fucci]{Guglielmo Fucci}
\address{Department of Mathematics, 
East Carolina University, 331 Austin Building, East Fifth Street,
Greenville, NC 27858-4353, USA}
\email{fuccig@ecu.edu}
\urladdr{http://myweb.ecu.edu/fuccig/}

\author[F.\ Gesztesy]{Fritz Gesztesy}
\address{Department of Mathematics, 
Baylor University, Sid Richardson Bldg., 1410 S.~4th Street, Waco, TX 76706, USA}
\email{Fritz$\_$Gesztesy@baylor.edu}
\urladdr{http://www.baylor.edu/math/index.php?id=935340}

\author[K.\ Kirsten]{Klaus Kirsten}
\address{Department of Mathematics, 
Baylor University, Sid Richardson Bldg., 1410 S.~4th Street, Waco, TX 76706, USA, and Mathematical
Reviews, American Mathematical Society, 416 4th Street, Ann Arbor, MI 48103, USA}
\email{Klaus$\_$Kirsten@baylor.edu}
\urladdr{http://www.baylor.edu/math/index.php?id=54012}

\author[J.\ Stanfill]{Jonathan Stanfill}
\address{Department of Mathematics, 
Baylor University, Sid Richardson Bldg., 1410 S.~4th Street,
Waco, TX 76706, USA}
\email{Jonathan$\_$Stanfill@baylor.edu}
\urladdr{http://sites.baylor.edu/jonathan-stanfill/}


\date{\today}
\thanks{Res. Math. Sci. {\bf 8}, No.~61 (2021), 46 pp.} 
\@namedef{subjclassname@2020}{\textup{2020} Mathematics Subject Classification}

\begin{abstract} 
The principal aim in this paper is to employ a recently developed unified approach to the computation of traces of resolvents and $\zeta$-functions to efficiently compute values of spectral $\zeta$-functions at positive integers associated to regular (three-coefficient) self-adjoint Sturm--Liouville differential expressions $\tau$. Depending on the underlying boundary conditions, we express the $\zeta$-function values in terms of a fundamental system of solutions of $\tau y = z y$ and their expansions about the spectral point $z=0$. Furthermore, we give the full analytic continuation of the $\zeta$-function through a Liouville transformation and provide an explicit expression for the $\zeta$-regularized functional determinant in terms of a particular set of this fundamental system of solutions.

An array of examples illustrating the applicability of these methods is provided, including regular Schr\"{o}dinger operators with zero, piecewise constant, and a linear potential on a compact interval.
\end{abstract}

\maketitle

\tableofcontents



\section{Introduction} \lb{s1} 

The principal motivation for this paper is to illustrate how a recently developed unified approach to the computation of Fredholm determinants, traces of resolvents, and $\zeta$-functions in \cite{GK19} can be used to efficiently compute certain values of spectral $\zeta$-functions associated to regular Sturm--Liouville operators as well as give the full analytic continuation of the $\zeta$-function through a Liouville transformation and finally provide an explicit expression for the $\zeta$-regularized functional determinant.

In Section \ref{s2} we begin by outlining the background for regular self-adjoint Sturm--Liouville operators on bounded intervals, that is, operators in $L^2((a,b);rdx)$ with separated and coupled boundary conditions and the associated spectral $\zeta$-functions. Under appropriate hypotheses on the Sturm--Liouville operator associated with three-coefficient differential expressions of the type $\tau=r^{-1}[-(d/dx)p(d/dx)+q]$, certain values of the spectral $\zeta$-function can be found via complex contour integration techniques to be equal to residues of explicit functions involving a canonical system of fundamental solutions $\phi(z,\dott,a)$ and $\theta(z,\dott,a)$ of $\tau y=zy$ for separated or coupled boundary conditions. Moreover, the zeros with respect to the parameter $z$ of $\phi$, $\theta$, and some of their (boundary condition dependent) linear combinations are precisely the eigenvalues corresponding to the underlying operator, including multiplicity.

In Section \ref{s3} we provide a series expansion for $\phi(z,\dott,a)$ and $\theta(z,\dott,a)$ about $z=0$ using the Volterra integral equation associated with the general three-coefficient regular self-adjoint Sturm--Liouville operator. This method leads to an expansion in powers of $z$ of the fundamental solutions and their $z$-derivative involving their values at $z=0$ and the appropriate Volterra Green's function. We also investigate the $|z|\to\infty$ asymptotic expansion of the characteristic function appearing in the complex integral representation of the spectral $\zeta$-function given in Section \ref{s2}. This asymptotic expansion is then exploited in order to construct the analytic continuation of the spectral $\zeta$-function and to obtain an explicit expression for the zeta regularized functional determinant.

Section \ref{s4} contains the main theorems that allow for the calculation of the values of spectral $\zeta$-functions of general regular Sturm--Liouville operators on bounded intervals as ratios of series expansions of (boundary condition dependent) solutions
of $\tau y = z y$ about $z=0$. In particular, we consider separated boundary conditions when zero is not an eigenvalue, or, when it is (necessarily) a simple eigenvalue, and coupled boundary conditions when either zero is not an eigenvalue, or, an eigenvalue of multiplicity (necessarily) at most two. (For more details in this context see \cite{GK19} as well as \cite[Ch. 3]{GZ21}, \cite[Sect. 8.4]{We80}, \cite[Sect. 13.2]{We03}, and \cite[Ch. 4]{Ze05}.)

We continue by providing some examples in Section \ref{s5} illustrating the main theorems and corollaries of Section \ref{s4} and the zeta regularized functional determinant given in Section \ref{s3}. In particular, we present the case of  Schr\"odinger operators with zero potential imposing Dirichlet, Neumann, periodic, antiperiodic, and Krein--von Neumann  boundary conditions. We then consider positive (piecewise) constant and negative constant potentials for Dirichlet boundary conditions, and finally the case of a linear potential.

Here we summarize some of the basic notation used in this manuscript. If $A$ is a linear operator mapping (a subspace of) a Hilbert space into another, then $\dom(A)$ and $\ker(A)$ denote the domain and the kernel (i.e., null space) of $A$. The spectrum, point spectrum, and resolvent set of a closed linear operator in a separable complex Hilbert space, $\cH$, will be denoted by $\sigma(\dott),\ \sigma_p(\dott),$ and $\rho(\dott)$ respectively. If $S$ is self-adjoint in $\cH$, the multiplicity of an eigenvalue $z_0\in\sigma_p(S)$ is denoted $m(z_0;S)$ (the geometric and algebraic multiplicities of $S$ coincide in this case). The proper setting for our investigations is the Hilbert space $L^2((a,b);rdx)$, which we will occasionally abbreviate as $L^2_r((a,b))$. The spectral $\zeta$-function of a self-adjoint linear operator $S$ is denoted 
by $\zeta(s;S)$. In addition, ${\tr}_{\cH}(T)$ denotes the trace of a trace class operator $T \in\cB_1(\cH)$ and 
$\det_\cH(I_{\cH} - T)$ the Fredholm determinant of $I_{\cH} - T$. 

For consistency of notation, throughout this manuscript we will follow the conventional notion that derivatives annotated with superscripts are understood as with respect to $x$ and derivatives with respect to $\xi $ will be abbreviated by 
$\dot\ =d/d \xi$. We also employ the notation $\bbN_0= \bbN \cup \{0\}$.

\section{Background on Self-Adjoint Regular Sturm--Liouville Operators} \lb{s2} 

In the first part of this section we briefly recall basic facts on 
regular Sturm--Liouville operators and their self-adjoint boundary conditions. This material is standard and well-known, hence we just refer to some of the standard monographs on this subject, such as, 
\cite[Sect.~6.3]{BHS20}, \cite[Ch.~3]{GZ21}, \cite[Sect.~II.5]{JR76}, 
\cite[Ch.~V]{Na68}, \cite[Sect.~8.4]{We80}, \cite[Sect.~13.2]{We03}, \cite[Ch.~4]{Ze05}. In the second part we discuss Fredholm determinants, traces of resolvents, and spectral $\zeta$-functions associated with these regular Sturm--Liouville problems. For background as well as relevant material in this context we refer to \cite{Am20}, \cite{Aw18}, 
\cite{BFK95}, \cite{BF60}, \cite{DU11}, \cite{Di61}, \cite{DD78}, \cite{FFG17}, \cite{Fo87}, \cite{Fo92}, \cite{FL19}, \cite{FL20}, \cite{FGK15}, \cite{GK19}, \cite{GKM-AS18}, \cite{HS18}, \cite{JP51}, \cite{Le98}, \cite{LT98}, \cite{LV11}, \cite{LS77}, \cite{MS18}, \cite{Mu98}, \cite{MCKB15}, \cite{OY12}, \cite{RS97}, \cite{Sp20}, \cite[Sects.~5.4, 5.5, 6.3]{Ta08}, \cite{Ve18}.

Throughout our discussion of regular Sturm--Liouville operators we make the following assumptions:

\begin{hypothesis} \lb{h2.1}
Let $(a,b) \subset \bbR$ be a finite interval and suppose that $p,q,r$ are $($Lebesgue\,$)$ measurable functions on $(a,b)$  
such that the following items $(i)$--$(iii)$ hold: \\[1mm] 
$(i)$ $r > 0$ a.e.~on $(a,b)$, $r \in L^1((a,b);dx)$. \\[1mm]
$(ii)$ $p > 0$ a.e.~on $(a,b)$, $1/p \in L^1((a,b);dx)$. \\[1mm]
$(iii)$ $q$ is real-valued a.e.~on $(a,b)$, $q \in L^1((a,b);dx)$.  
\end{hypothesis}

Given Hypothesis \ref{h2.1}, we now study Sturm--Liouville operators associated with the general, 
three-coefficient differential expression $\tau$ of the type,
\begin{align}\lb{2.1}
\tau=\f{1}{r(x)}\left[-\f{d}{dx}p(x)\f{d}{dx} + q(x)\right] \, \text{ for a.e.~$x\in(a,b) \subseteq \R$.} 
\end{align} 

We start with the notion of minimal and maximal $\Lr$-realizations associated with the regular differential 
expression $\tau$ on the finite interval $(a,b) \subset \bbR$.  Here, and elsewhere throughout this manuscript, the inner product in $\Lr$ is defined by
\begin{align}
(f,g)_{\Lr} = \int_a^b r(x) dx \, \ol{f(x)}g(x), \quad f,g\in \Lr.
\end{align}

Assuming Hypothesis \ref{h2.1}, the differential expression $\tau$ of the form \eqref{2.1} on the finite interval 
$(a,b) \subset \bbR$ is called {\it regular on} $[a,b]$. The corresponding \textit{maximal operator} $T_{max}$ 
in $\Lr$ associated with $\tau$ is defined by 
\begin{align}
&T_{max} f = \tau f,     \no 
\\
& f \in \dom(T_{max})=\big\{g\in\Lr  \, \big| \,  g,g^{[1]}\in\AC; \\
& \hspace*{6cm} \tau g\in\Lr\big\},   \no 
\end{align}
and the corresponding \textit{minimal operator} $T_{min}$ in $\Lr$ associated with $\tau$ is given by 
\begin{align}
&T_{min} f = \tau f, \no 
\\
& f \in \dom(T_{min})=\big\{g\in\Lr  \, \big| \,  g,g^{[1]}\in\AC; \\ 
&\hspace*{3cm} g(a)=g^{[1]}(a)=g(b)=g^{[1]}(b)=0; \; \tau g\in\Lr\big\}.  \no  
\end{align}

Here (with $\prime := d/dx$)
\begin{align}
y^{[1]}(x) = p(x) y'(x),
\end{align}
denotes the first quasi-derivative of a function $y$ on $(a,b)$, assuming that $y, py' \in AC_{loc}((a,b))$. 

Assuming Hypothesis \ref{h2.1} so that $\tau$ is regular on $[a,b]$, the following is well-known (see, e.g., 
\cite[Sect.~6.3]{BHS20}, \cite[Sect.~3.2]{GZ21}, \cite[Sect.~II.5]{JR76}, \cite[Ch.~V]{Na68}, 
\cite[Sect.~8.4]{We80}, 
\cite[Sect.~13.2]{We03}, \cite[Ch.~4]{Ze05}): $T_{min}$ is a densely defined, closed operator in $\Lr$, moreover, $T_{max}$ is densely defined and closed in $\Lr$, and 
\begin{align}
T_{min}^* = T_{max}, \quad T_{min} = T_{max}^*.
\end{align}
Moreover, $T_{min} \subset T_{max} = T_{min}^*$, and hence $T_{min}$ is symmetric, while 
$T_{max}$ is not. 

The next theorem describes all self-adjoint extensions of $T_{min}$ (cf., e.g., \cite[Sect.~13.2]{We03}, \cite[Ch.~4]{Ze05}).

\begin{theorem}\lb{t2.2}
Assume Hypothesis \ref{h2.1} so that $\tau$ is regular on $[a,b]$. Then the following items $(i)$--$(iii)$ hold: \\[1mm] 
$(i)$ All self-adjoint extensions $T_{\al,\be}$ of $T_{min}$ with separated boundary conditions are of the form
\begin{align}
& T_{\al,\be} f = \tau f, \quad \al,\be\in[0,\pi),   \no  \\
& f \in \dom(T_{\al,\be})=\big\{g\in\dom(T_{max}) \, \big| \, g(a)\cos(\al)+g^{[1]}(a)\sin(\al)=0;  \lb{2.7} \\ 
& \hspace*{5.6cm} g(b)\cos(\be)-g^{[1]}(b)\sin(\be) = 0 \big\}.    \no  
\end{align}
Special cases: $\al=0$ $($i.e., $g(a)=0$$)$ is called the Dirichlet boundary condition at $a$; $\al=\f\pi2$,  
$($i.e., $g^{[1]}(a)=0$$)$ is called the Neumann boundary condition at $a$ $($analogous facts hold at the 
endpoint $b$$)$. \\[1mm]
$(ii)$ All self-adjoint extensions $T_{\varphi,R}$ of $T_{min}$ with coupled boundary conditions are of the type
\begin{align}
\begin{split} 
& T_{\varphi,R} f = \tau f,    \\
& f \in \dom(T_{\varphi,R})=\bigg\{g\in\dom(T_{max}) \, \bigg| \begin{pmatrix}g(b)\\g^{[1]}(b)\end{pmatrix} 
= e^{i\varphi}R \begin{pmatrix}
g(a)\\g^{[1]}(a)\end{pmatrix} \bigg\},     \lb{2.8} 
\end{split}
\end{align}
where $\varphi\in[0,\pi)$, and $R$ is a real $2\times2$ matrix with $\det(R)=1$ 
$($i.e., $R \in SL(2,\bbR)$$)$.
Special cases: $\varphi = 0$, $R=I_2$ $($i.e., $g(b)=g(a)$, $g^{[1]}(b)=g^{[1]}(a)$$)$ are called {\it periodic boundary conditions}; 
similarly, $\varphi = 0$, $R= - I_2$ $($i.e., $g(b)=-g(a)$, $g^{[1]}(b)=-g^{[1]}(a)$$)$ are called {\it antiperiodic boundary conditions}. 
\\[1mm] 
$(iii)$ Every self-adjoint extension of $T_{min}$ is either of type $(i)$ $($i.e., separated\,$)$ or of type 
$(ii)$ $($i.e., coupled\,$)$.
\end{theorem}

Next we state some of the most pertinent concepts and results summarized from \cite{GK19} (in particular, Section 3) and will then illustrate how this permits one to effectively calculate certain values for the spectral $\zeta$-functions of the regular Sturm--Liouville operators considered.

For this purpose we introduce the fundamental system of solutions $\theta(z,x,a)$, $\phi(z,x,a)$ of $\tau y=z y$ defined by
\begin{align}\lb{2.9}
\theta(z,a,a)=\phi^{[1]}(z,a,a)=1,\quad \theta^{[1]}(z,a,a)=\phi(z,a,a)=0,
\end{align}
such that
\begin{align}
    W(\theta(z,\dott,a),\phi(z,\dott,a))=1,
\end{align}
noting that for fixed $x,$ each is entire with respect to $z$. Here the Wronskian of $f$ and $g$, for $f,g\in\ACl$, is defined by
\begin{align}
W(f,g)(x) = f(x)g^{[1]}(x) - f^{[1]}(x)g(x).
\end{align}

Furthermore, we introduce the boundary values for $g,g^{[1]}\in\AC$, see \cite[Ch.~I]{Na67}, \cite[Sect. 3.2]{Ze05},
\begin{align}
\begin{split}
& U_{\a,\b,1}(g)=g(a)\cos(\al)+g^{[1]}(a)\sin(\al),\\
& U_{\a,\b,2}(g)=g(b)\cos(\be)-g^{[1]}(b)\sin(\be),
\end{split}
\end{align}
in the case ($i$) of separated boundary conditions in Theorem \ref{t2.2}, and
\begin{align}
\begin{split}
& V_{\varphi,R,1}(g)=g(b)-e^{i\varphi}R_{11}g(a)-e^{i\varphi}R_{12}g^{[1]}(a),\\
& V_{\varphi,R,2}(g)=g^{[1]}(b)-e^{i\varphi}R_{21}g(a)-e^{i\varphi}R_{22}g^{[1]}(a),
\end{split}
\end{align}
in the case ($ii$) of coupled boundary conditions in Theorem \ref{t2.2}. Moreover, we define the \emph{characteristic functions}
\begin{align}
F_{\a,\b}(z)=\det\begin{pmatrix}U_{\a,\b,1}(\theta(z,\dott,a))& U_{\a,\b,1}(\phi(z,\dott,a))\\ U_{\a,\b,2}(\theta(z,\dott,a)) & U_{\a,\b,2}(\phi(z,\dott,a))\end{pmatrix},
\end{align}
and
\begin{align}
F_{\varphi,R}(z)=\det\begin{pmatrix}V_{\varphi,R,1}(\theta(z,\dott,a))& V_{\varphi,R,1}(\phi(z,\dott,a))\\ V_{\varphi,R,2}(\theta(z,\dott,a)) & V_{\varphi,R,2}(\phi(z,\dott,a))\end{pmatrix}.
\end{align}

\noindent 
{\bf Notational Convention.} {\it To describe all possible self-adjoint boundary conditions associated with self-adjoint extensions of $T_{min}$ effectively, we will frequently employ the notation $T_{A,B}$, $F_{A,B}$, 
$\lambda_{A,B,j}$, $j \in J$, etc., where $A,B$ represents $\alpha,\beta$ in the case of separated boundary conditions and $\varphi,R$ in the context of coupled boundary conditions.} 

By construction, eigenvalues of $T_{A,B}$ are determined via $F_{A,B}(z)=0$, with multiplicity of eigenvalues of $T_{A,B}$ corresponding to multiplicity of zeros of $F_{A,B}$, and $F_{A,B}(z)$ is entire with respect to $z$. In particular, for $T_{\a,\b}$, that is, for separated boundary conditions, one has 
\begin{align}\lb{2.16}
\begin{split}
F_{\a,\b}(z)&= \cos(\a)[-\sin(\b)\ \phi^{[1]}(z,b,a)+\cos(\b)\ \phi(z,b,a)] \\
&\quad -\sin(\a)[-\sin(\b)\ \theta^{[1]}(z,b,a)+\cos(\b)\ \theta(z,b,a)],\quad \a,\b\in[0,\pi),
\end{split}
\end{align}
and for $T_{\varphi,R}$, that is, for coupled boundary conditions, one has for $\varphi\in[0,\pi)$ and $R\in SL(2,\R)$,
\begin{align}\lb{2.17}
\no  F_{\varphi,R}(z)&= e^{i\varphi}\big(R_{12}\theta^{[1]}(z,b,a)-R_{22}\theta(z,b,a)+R_{21}\phi(z,b,a)-R_{11}\phi^{[1]}(z,b,a)\big)\\
&\quad +e^{2i\varphi}+1.
\end{align}

Next we will demonstrate that $F_{A,B}(\dott)$ is an entire function of order $1/2$ and finite type, independent of the boundary conditions chosen. This result is used when considering convergence of the complex contour integral representation of the spectral $\zeta$-function for large values of the spectral parameter $z$.

For this purpose we recall the following facts (see, e.g., \cite[Ch.~2]{Bo54}, \cite[Ch.~I]{Le80}): Supposing that $F(\dott)$ is entire, one introduces 
\begin{equation}
M_F(R) = \sup_{|z|=R} |F(z)|, \quad R \in [0,\infty). 
\end{equation} 
Then the order $\rho_F$ of $F$ is defined by
\begin{equation}
\rho_F = \limsup_{R \to \infty} \ln(\ln(M_F(R)))/\ln(R) \in [0,\infty) \cup \{\infty\}.
\end{equation}
In addition, if $\rho_F > 0$, the type $\tau_F$ of $F$ is defined as 
\begin{equation}
\tau_F = \limsup_{R \to \infty} \ln(M_F(R))/R^{\rho_F} \in [0,\infty) \cup \{\infty\},
\end{equation}
and, in obvious notation, $F$ is called of order $\rho_F > 0$ and of finite type $\tau_F$ if $\tau_F \in [0,\infty)$.

Thus, $F$ is of finite order $\rho_F \in [0,\infty)$ if and only if for every $\varepsilon > 0$, but for no 
$\varepsilon < 0$, 
\begin{equation}
M_F(R) \underset{R \to \infty}{=} \Oh\big(\exp\big(R^{\rho_F + \varepsilon}\big)\big),  
\end{equation}
and $F$ is of positive and finite order $\rho_F \in (0,\infty)$ and finite type $\tau_F \in [0,\infty)$ if and only if 
for every $\varepsilon > 0$, but for no $\varepsilon < 0$, 
\begin{equation}
M_F(R) \underset{R \to \infty}{=} \Oh\big(\exp\big((\tau_F + \varepsilon) R^{\rho_F}\big)\big). 
\end{equation}

By definition, if $F_j$ are entire of order $\rho_j$, $j=1,2$, then the order of $F_1 F_2$ does not exceed the larger of $\rho_1$ and $\rho_2$. 

For $F$ entire we also introduce the zero counting function 
\begin{equation}
N_F(R) = \# \big(Z_F \cap \ol{D(0;R)}\big), \quad R \in (0,\infty), 
\end{equation}
where $\#$ denotes cardinality and $Z_F$ represents the set of zeros of $F$ counting multiplicity (i.e., $N_F(R)$ counts the number of zeros of $F$ in the closed disk of radius $R >0$ centered at the origin).  

\begin{remark}\lb{r2.3}
Assuming Hypothesis \ref{h2.1}, then all solutions $\psi(z,\dott)$ of the regular Sturm--Liouville problem 
$(\tau y)(z,x)=zy(z,x)$, $z \in \bbC$, $x \in [a,b]$, satisfying $z$-independent initial conditions
\begin{equation}
\psi(z,x_0) = c_0, \quad \psi^{[1]}(z,x_0) = c_1,
\end{equation}
for some $x_0 \in [a,b]$ and some $(c_0,c_1) \in \bbC^2$, together with $\psi^{[1]}(z,\dott)$, for any 
fixed $x \in [a,b]$, are entire functions of $z$ of order at most $1/2$. Indeed, as shown in 
\cite[Sect.\ 8.2]{At64} (see also \cite{Mi83}, \cite[Theorem~2.5.3]{Ze05}), upon employing a Pr\"ufer-type transformation, one obtains 
\begin{align}
& |z| |\psi(z,x)|^2 + \big|\psi^{[1]}(z,x)\big|^2 \leq C(x_0) \exp\bigg(|z|^{1/2} 
\int_{\min(x_0,x)}^{\max(x_0,x)} dt \, \big[|p(t)|^{-1} + |r(t)|\big]    \no \\
& \hspace*{3cm} + |z|^{-1/2} \int_{\min(x_0,x)}^{\max(x_0,x)} dt \, |q(t)|\bigg),  \quad z \in \bbC, \; x_0, x \in [a,b].     \lb{2.25} 
\end{align}  
In particular, \eqref{2.16} and \eqref{2.17} yield that $F_{A,B}$ is an entire function of order at most $1/2$ for any self-adjoint boundary condition represented by $A,B$, that is,
\begin{equation}
\rho_{F_{A,B}} \leq 1/2.      \lb{2.26} 
\end{equation} 

Given Hypothesis \ref{h2.1}, one infers that $T_{A,B} \geq \Lambda_{A,B} I_{L^2_r((a,b))}$ for some 
$\Lambda_{A,B} \in \bbR$, with purely discrete spectrum, and hence 
$Z_{F_{A,B}}(R) \subset [\Lambda_{A,B},R]$  the elements of $Z_{F_{A,B}}(R)$ being precisely the 
eigenvalues of $T_{A,B}$ in the interval $[\max(-R, \Lambda_{A,B}),R]$. Employing the theory of Volterra operators in Hilbert spaces (and under some additional lower boundedness hypotheses\footnote{Upon closer inspection, the additional condition stated on \cite[p.~305, 306]{GK70} just ensures lower semiboundedness of $T_{A,B}$, which is independently known to hold in our present scalar context.} on $q$) in \cite[Chs.~VI, VII]{GK70}, alternatively, using oscillation theoretic methods in \cite{AM87}, it is shown that the eigenvalue counting function $N_{F_{A,B}}$ associated with $T_{A,B}$ satisfies 
\begin{equation}
N_{F_{A,B}} (\lambda) \underset{\lambda \to \infty}{=} \pi^{-1} \int_a^b dx \, [r(x)/p(x)]^{1/2} 
\lambda^{1/2} [1 + \oh(1)]. 
\lb{2.27}
\end{equation}
Ignoring finitely many nonpositive eigenvalues of $T_{A,B}$, equivalently, splitting off the factors in the infinite product representation associated with nonpositive zeros of $F_{A,B}$, that is, replacing $F_{A,B}$ by 
\begin{equation}
\wti F_{A,B} (z) = C_{A,B} \prod_{\substack{j \in \bbN, \\ \lambda_{A,B,j} > 0}} [1 - (z/\lambda_{A,B,j})]  
\end{equation} 
with 
\begin{equation}
N_{\wti F_{A,B}} (\lambda) \underset{\lambda \to \infty}{=} \pi^{-1} \int_a^b dx \, [r(x)/p(x)]^{1/2} 
\lambda^{1/2} [1 + \oh(1)],  
\end{equation}
implies (cf.\ \cite[Theorem~4.1.1]{Bo54}, \cite{Ti26}, \cite{Ti27}),
\begin{equation}
\ln\big(\wti F_{A,B}(\lambda)\big) \underset{\lambda \to \infty}{=} \int_a^b dx \, [r(x)/p(x)]^{1/2} 
\lambda^{1/2} [1 + \oh(1)].  
\end{equation}
Thus, 
\begin{equation}
\rho_{F_{A,B}} = \rho_{\wti F_{A,B}} \geq 1/2, 
\end{equation}
and hence by \eqref{2.26}, 
\begin{equation}
\rho_{F_{A,B}} = 1/2. 
\end{equation}
Moreover, by \eqref{2.25}, $F_{A,B}$ is of order 1/2 and finite type. Finally, we also mention that \eqref{2.27} implies that 
\begin{equation}
\lambda_{A,B,j} \underset{j \to \infty}{=} \bigg[\int_a^b dx \, [r(x)/p(x)]^{1/2}\bigg]^{-2} \pi^2 j^2 [1 + \oh(1)]
\end{equation}
(cf.\ also the discussion in \cite[Sects.~1.11, 9.1]{RSS94}, \cite[Sect.~4.3]{Ze05}). 
\hfill $\diamond$
\end{remark}

The following theorem (see \cite[Thm.~3.4]{GK19}) directly relates the function $F_{A,B}$ to Fredholm determinants and traces (see \cite[Ch.~IV]{GK69}, \cite[Sect.~XIII.17]{RS78}, \cite{Si77}, 
\cite[Ch.~3]{Si05}, \cite[Ch.~3]{Si15} for background).

\begin{theorem}\lb{t2.4}
Assume Hypothesis \ref{h2.1} and denote by $T_{\a,\b}$ and $T_{\varphi,R}$ the self-adjoint extensions of $T_{min}$ as described in cases $(i)$ and $(ii)$ of Theorem \ref{t2.2}, respectively. \\[1mm]
$(i)$ Suppose $z_0\in\rho(T_{\a,\b})$, then
\begin{align}
\begin{split}
\det&_{L_r^2((a,b))}\big(I_{L_r^2((a,b))}-(z-z_0)(T_{\a,\b}-z_0I_{L_r^2((a,b))})^{-1}\big)\\
&=F_{\a,\b}(z)/F_{\a,\b}(z_0),\quad z\in\C.
\end{split}
\end{align}
In particular,
\begin{align}
    \tr_{L_r^2((a,b))}\big((T_{\a,\b}-zI_{L_r^2((a,b))})^{-1}\big)=-(d/dz)\ln(F_{\a,\b}(z)),\quad z\in\rho(T_{\a,\b}).
\end{align}
$(ii)$ Suppose $z_0\in\rho(T_{\varphi,R})$, then
\begin{align}
\begin{split}
\det&_{L_r^2((a,b))}\big(I_{L_r^2((a,b))}-(z-z_0)(T_{\varphi,R}-z_0I_{L_r^2((a,b))})^{-1}\big)\\
&=F_{\varphi,R}(z)/F_{\varphi,R}(z_0),\quad z\in\C.
\end{split}
\end{align}
In particular,
\begin{align}
    \tr_{L_r^2((a,b))}\big((T_{\varphi,R}-zI_{L_r^2((a,b))})^{-1}\big)=-(d/dz)\ln(F_{\varphi,R}(z)),\quad z\in\rho(T_{\varphi,R}).
\end{align}
\end{theorem}

Given these preparations, we let $T_{A,B}$ denote the self-adjoint extension of $T_{min}$ with either separated 
($T_{\a,\b}$) or coupled ($T_{\varphi,R}$) boundary conditions as described in cases $(i)$ and $(ii)$ of Theorem \ref{t2.2}. One recalls (see, e.g., \cite{GK19}), the spectral $\zeta$-function of the operator, $T_{A,B}$, is defined as
\begin{align}\lb{2.39}
\zeta(s;T_{A,B}):=\sum_{\underset{\lambda_j\neq 0}{j\in J}} \lambda_{A,B,j}^{-s}, 
\end{align}
with $J\subset \Z$ an appropriate index set counting eigenvalues according to their multiplicity and 
$\Re(s)>0$ sufficiently large such that \eqref{2.39} converges absolutely. Applying Theorem \ref{t2.4}, it was shown in 
\cite{GK19} that for $\Re(s)>0$ sufficiently large,
\begin{align}\lb{2.40}
\begin{split}
\zeta(s;T_{A,B})&=\dfrac{1}{2\pi i}\ointctrclockwise_\gamma dz\ z^{-s}\bigg(\dfrac{d}{dz}\ln (F_{A,B}(z))-z^{-1}m(0;T_{A,B})\bigg)\\
&=\dfrac{1}{2\pi i}\ointctrclockwise_\gamma dz\ z^{-s}\bigg(\dfrac{d}{dz}\ln (F_{A,B}(z))-z^{-1}m_0\bigg), 
\end{split}
\end{align}
where $m(0;T_{A,B})=m_0$ is the multiplicity of zero as an eigenvalue of $T_{A,B}$ and $\gamma$ is a simple contour enclosing $\sigma(T_{A,B})\backslash\{0\}$ in a counterclockwise manner so as to dip under (and hence avoid) the point 0 (cf. Figure ~\ref{f1}). Here, following \cite{KM03} (see also \cite{KM04}), we take 
\begin{align}
    R_\psi=\{z=te^{i\psi} \, | \, t\in [0,\infty)\},\quad \psi\in (\pi/2,\pi),      \lb{2.40a} 
\end{align}
to be the branch cut of $z^{-s}$, and, once again, eigenvalues will be determined via $F_{A,B}(z)=0$, with the multiplicity of eigenvalues of $T_{A,B}$ corresponding to the multiplicity of zeros of $F_{A,B}$.

\begin{figure}[H]
\vspace{.8cm}
{\setlength{\unitlength}{1.0cm}
\begin{picture}(10,4.5)
\thicklines

\put(0,1.5){\vector(1,0){10}} \put(5.0,0){\vector(0,1){5}}
\put(10.0,1.5){\oval(9.4,1)[tl]} \put(5,1.5){\oval(0.6,0.5)[b]}
\put(4.4,1.5){\oval(0.6,1)[tr]} \put(4.4,1.5){\oval(3.0,1)[l]}
\put(7.15,2){\vector(-1,0){0}}
\put(4.4,1){\line(1,0){5.6}}
\put(5.0,1.5){\line(-1,2){1.5}} \put(.7,4.8){{\bf The cut $\boldsymbol{R_{\psi}}$ for
$\boldsymbol{z^{-s}}$}} \multiput(5.6,1.5)(.4,0){10}{\circle*{.15}}
\multiput(3.5,1.5)(.4,0){3}{\circle*{.15}}
\put(8.0,4.5){{\bf $\boldsymbol{z}$-plane}}
\put(7.15,2.3){{\bf $\boldsymbol{\gamma}$}}
\end{picture}}
\caption{Contour $\gamma$ in the complex $z$-plane.}
\lb{f1}

\vspace{.5cm}
\begin{multicols}{2}
{\setlength{\unitlength}{.6cm}
\begin{picture}(10,5)
\thicklines

\put(0,1.5){\vector(1,0){10}}
\put(5.0,0){\vector(0,1){5}}
\put(5,1.5){\oval(0.6,0.5)[b]}
\put(4.7,1.5){\line(-1,2){1.5}}
\put(3.7,3.45){\vector(-1,2){0}}
\put(5.3,1.5){\line(-1,2){1.5}}
\put(4.3,3.5){\vector(1,-2){0}}
\multiput(5.0,1.5)(-.45,.89){4}{\line(-1,2){.29289}}
\put(0,5.1){{\bf The cut $\boldsymbol{R_{\psi}}$ for $\boldsymbol{z^{-s}}$}}
\multiput(5.6,1.5)(.4,0){10}{\circle*{.15}}
\multiput(3.5,1.5)(.4,0){3}{\circle*{.15}}
\put(8.0,4.5){{\bf $\boldsymbol{z}$-plane}}
\put(5.35,1.85){{\bf $\boldsymbol{\gamma}$}}
\end{picture}}
\caption{Deforming $\gamma$.}
\lb{f2}

\vspace{.3cm}
{\setlength{\unitlength}{.6cm}
\begin{picture}(10,5)
\thicklines

\put(0,1.5){\vector(1,0){10}}
\put(5.0,0){\vector(0,1){5}}
\put(5.0,1.5){\circle{.75}}
\put(5.1,1.1){\vector(-1,-1){0}}
\multiput(5.6,1.5)(.4,0){10}{\circle*{.15}}
\multiput(3.5,1.5)(.4,0){3}{\circle*{.15}}
\put(8.0,4.5){{\bf $\boldsymbol{z}$-plane}}
\put(5.35,1.85){{\bf $\boldsymbol{C_\varepsilon}$}}
\end{picture}}
\caption{Contour $C_\varepsilon$.}
\lb{f3}
\end{multicols}
\end{figure}

To continue the computation of \eqref{2.40} and deform the contour $\gamma$ as to ``hug'' the branch cut $R_\psi$ (cf. Figure \ref{f2}) requires knowledge of the asymptotic behavior of $F_{A,B}(z)$ as $|z|\to\infty$, which in turn demands 
$\Re(s)>1/2$ for large-$z$ convergence (cf.\ Remark \ref{r2.3}). 
Furthermore, if one is interested in the calculation of the value of the spectral zeta function at positive integers, the following method provides a very simple way of obtaining those values. In fact, by letting $s=n$, $n\in\N$, in \eqref{2.40}, one no longer needs a branch cut for the fractional powers of $z^{-s}$ given in Figures \ref{f1} and \ref{f2}. This reduces the integral along the curve $\gamma$ to a clockwise oriented integral along the circle $C_\varepsilon$, centered at zero with radius $\varepsilon>0$ (cf. Figure \ref{f3}). Letting $s=n$ also ensures that $m_0$ (the multiplicity of zero as an eigenvalue of $T_{A,B}$) does not contribute to the integral in \eqref{2.40}. Hence,
\begin{align}\lb{2.42}
\begin{split}
\zeta(n;T_{A,B})&=-\dfrac{1}{2\pi i}\ointctrclockwise_{C_\varepsilon}dz\ z^{-n}\dfrac{d}{dz}\ln (F_{A,B}(z))\\
&=-\Res \left[ z^{-n}\dfrac{d}{dz}\ln (F_{A,B}(z));\ z=0 \right],\quad n\in\N.
\end{split}
\end{align}

Thus, determining an expansion of $F_{A,B}(z)$ about $z=0$ enables one to effectively compute 
$\zeta(n;T_{A,B})$. In 
addition, by \eqref{2.16}, \eqref{2.17}, $F_{A,B}(z)$ is a linear combination of $\theta$, $\theta^{[1]}$, 
$\phi$, and $\phi^{[1]}$ for each boundary condition considered, so it suffices to find the expansion of each of these functions individually.

\section{Expansion in \textit{z} for Fundamental Solutions, Asymptotic Expansion, and the Zeta Regularized Functional Determinant} \lb{s3}

\subsection{Expansion in \textit{z} for Fundamental Solutions} \lb{s3.1}
\hfill

Assuming Hypothesis \ref{h2.1} throughout this section, we discuss next the expansion in $z$ about $z=0$ for the solutions $\phi(z,\dott,a)$ and $\theta(z,\dott,a)$ of $\tau y=zy$,

\begin{align}
\phi(z,x,a)&= \phi(0,x,a) + z\int_a^x r(x')dx' \, g(0,x,x') \phi(z,x',a),  \lb{3.1} \\
\theta(z,x,a)&= \theta(0,x,a) + z\int_a^x r(x')dx' \, g(0,x,x') \theta(z,x',a),    \lb{3.2} \\
& \hspace*{4.8cm} z \in \bbC, \, x \in [a,b],    \no 
\end{align}
employing the following expression for the Volterra Green's function
\begin{align}
g(0,x,x') = \theta(0,x,a)\phi(0,x',a)-\theta(0,x',a)\phi(0,x,a), \quad x, x' \in [a,b].
\end{align}
That \eqref{3.1} and \eqref{3.2} indeed represent solutions of $\tau y=zy$ is clear from applying $\tau$ to either side, moreover, the initial conditions \eqref{2.9} are readily verified. 

Iterating these integral equations establishes the power series expansions
\begin{align}\lb{3.4}
    \phi(z,x,a)=\sum_{m=0}^\infty z^m\phi_m(x),\quad z\in\C, \; x \in [a,b], 
\end{align}
where
\begin{align}
\begin{split}
\phi_0(x)&= \phi(0,x,a),\\
\phi_1(x)&= \int_a^x r(x_1)dx_1\ g(0,x,x_1)\phi(0,x_1,a),\\
\phi_k(x)&= \int_a^x r(x_1)dx_1\ g(0,x,x_1)\int_a^{x_1} r(x_2)dx_2\ g(0,x_1,x_2)\dots\\
& \quad \dots\int_a^{x_{k-1}} r(x_k)dx_k\ g(0,x_{k-1},x_k)\phi(0,x_k,a),\quad k\in\N,
\end{split}
\end{align}
and 
\begin{align}\lb{3.6}
    \theta(z,x,a)=\sum_{m=0}^\infty z^m\theta_m(x),\quad z\in\C, \; x \in [a,b], 
\end{align}
where
\begin{align}
\begin{split}
\theta_0(x)&= \theta(0,x,a),\\
\theta_1(x)&= \int_a^x r(x_1)dx_1\ g(0,x,x_1)\theta(0,x_1,a),\\
\theta_k(x)&= \int_a^x r(x_1)dx_1\ g(0,x,x_1)\int_a^{x_1} r(x_2)dx_2\ g(0,x_1,x_2)\dots\\
&\quad \dots\int_a^{x_{k-1}} r(x_k)dx_k\ g(0,x_{k-1},x_k)\theta(0,x_k,a),\quad k\in\N.
\end{split}
\end{align}

Analogously one obtains 
\begin{align}\lb{3.8}
    \phi^{[1]}(z,x,a)=\sum_{m=0}^\infty z^m\phi^{[1]}_m(x),\quad z\in\C, \; x \in [a,b], 
\end{align}
where
\begin{align}
\begin{split}
\phi^{[1]}_0(x)&= \phi^{[1]}(0,x,a),\\
\phi^{[1]}_1(x)&= \int_a^x r(x_1)dx_1\ g^{[1]}(0,x,x_1)\phi(0,x_1,a),\\
\phi^{[1]}_k(x)&= \int_a^x r(x_1)dx_1\ g^{[1]}(0,x,x_1)\int_a^{x_1} r(x_2)dx_2\ g(0,x_1,x_2)\dots\\
& \quad \dots\int_a^{x_{k-1}} r(x_k)dx_k\ g(0,x_{k-1},x_k)\phi(0,x_k,a),\quad k\in\N,
\end{split}
\end{align}
using the abbreviation
\begin{align}
g^{[1]}(0,x,x_1)=\theta^{[1]}(0,x,a)\phi(0,x_1,a)-\theta(0,x_1,a)\phi^{[1]}(0,x,a).
\end{align}
Similarly, one finds from \eqref{3.6}
\begin{align}\lb{3.11}
    \theta^{[1]}(z,x,a)=\sum_{m=0}^\infty z^m\theta^{[1]}_m(x),\quad z\in\C, \; x \in [a,b], 
\end{align}
where
\begin{align}
\begin{split}
\theta^{[1]}_0(x)&= \theta^{[1]}(0,x,a),\\
\theta^{[1]}_1(x)&= \int_a^x r(x_1)dx_1\ g^{[1]}(0,x,x_1)\theta(0,x_1,a),\\
\theta^{[1]}_k(x)&= \int_a^x r(x_1)dx_1\ g^{[1]}(0,x,x_1)\int_a^{x_1} r(x_2)dx_2\ g(0,x_1,x_2)\dots\\
& \quad \dots\int_a^{x_{k-1}} r(x_k)dx_k\ g(0,x_{k-1},x_k)\theta(0,x_k,a),\quad k\in\N.
\end{split}
\end{align}

\subsection{Asymptotic Expansion of the Characteristic Function} \lb{s3.2}
\hfill

Next we investigate the $|z|\to\infty$ asymptotic expansion of the function $F_{A,B}(z)$ in order to provide an analytic continuation of the spectral $\zeta$-function, $\zeta(s;T_{A,B})$, and compute the zeta regularized functional determinant. We first strengthen Hypothesis \ref{h2.1} by introducing the following assumptions on $p,q,r$ following \cite[Sect. 3]{GK19}. These additional 
assumptions\footnote{The original archive submission and the published version of this article in Res.~Math.~Sci.~{\bf 8}, No.~61 (2021), 46 pp., assumed in addition that $1/r\in L^\infty((a,b);dx)$ in item $(i)$ of Hypothesis~3.1. However, this additional assumption is superfluous. Its origin traces back to the same superfluous assumption made in Hypothesis~3.9\,$(i)$ in \cite{GK19}. (Incidentally, we note that the assumption, for some $\varepsilon >0$, $pr \geq \varepsilon$ on $[a,b]$, was missed in \cite[Hypothesis~3.9\,$(iv)$]{GK19}.)} are necessary in order to perform a Liouville-type transformation.

\begin{hypothesis} \lb{h3.1}
Let $(a,b) \subset \bbR$ be a finite interval and suppose that $p,q,r$ are $($Lebesgue\,$)$ measurable functions on $(a,b)$ such that the following items $(i)$--$(iv)$ hold: \\[1mm] 
$(i)$ $r > 0$ a.e.~on $(a,b)$, $r \in L^1((a,b);dx)$. \\[1mm] 
$(ii)$ $p > 0$ a.e.~on $(a,b)$, $1/p \in L^1((a,b);dx)$. \\[1mm]
$(iii)$ $q$ is real-valued a.e.~on $(a,b)$, $q \in L^1((a,b);dx)$. \\[1mm]
$(iv)$ $pr$ and $(pr)'/r$ are absolutely continuous on $[a,b]$, and for some $\varepsilon >0$, $pr \geq \varepsilon$ 
on $[a,b]$.
\end{hypothesis}

The variable transformations (cf. \cite[p. 2]{LS75}), 
\begin{align}
&\xi(x)=\dfrac{1}{c}\int_a^x dt\ [r(t)/p(t)]^{1/2},\quad \xi(x) \in [0,1] \, \text{ for } \, x \in [a,b],   \lb{3.13}\\
& \xi'(x) = c^{-1} [r(x)/p(x)]^{1/2} > 0 \, \text{ a.e.~on $(a,b)$,}    \lb{3.14} \\
&u(z,\xi)=[p(x(\xi))r(x(\xi))]^{1/4}y(z,x(\xi)),     \lb{3.15} 
\end{align}
with $c > 0$ given by
\begin{align}
c=\int_a^b dt\ [r(t)/p(t)]^{1/2},
\end{align}
transform the Sturm--Liouville problem $(\tau y(z,\dott))(x)=zy(z,x)$, $x \in (a,b)$, into
\begin{align}\lb{3.17}
- \ddot u (z,\xi)+V(\xi)u(z,\xi)=c^2 zu(z,\xi), \quad \xi \in (0,1), 
\end{align}
and abbreviating 
\begin{equation} 
\nu(\xi)=[p(x(\xi))r(x(\xi))]^{1/4},
\end{equation}
one verifies that
\begin{align}
\begin{split}
V(\xi)&=\dfrac{\ddot \nu (\xi)}{\nu(\xi)}+c^2\dfrac{q(x)}{r(x)}\\
&=-\dfrac{c^2}{16}\dfrac{1}{p(x)r(x)}\left[\dfrac{(p(x)r(x))^{\prime}}{r(x)}\right]^2+\dfrac{c^2}{4}\dfrac{1}{r(x)} \left[\dfrac{(p(x)r(x))^{\prime}}{r(x)}\right]^{\prime} + c^2\dfrac{q(x)}{r(x)}, 
\end{split}
\end{align}
and 
\begin{equation} 
V\in L^1((0,1);d\xi), 
\end{equation} 
as guaranteed by Hypothesis \ref{h3.1}.

In order to construct the asymptotic expansion of $F_{A,B}(z)$ we begin by assuming Hypothesis \ref{h3.1}, but note that throughout the construction of the expansion stronger assumptions will be necessary, all of which will be addressed once the final asymptotic expansion is given.

When applying the Liouville transformation the boundary conditions undergo a similar transformation. 
In fact, setting 
\begin{equation} 
Q(\xi)=[(pr)'/r](x(\xi)) 
\end{equation} 
one can write
\begin{align}\lb{3.22}
\begin{pmatrix}u(z,\xi)\\\dot u(z,\xi)\end{pmatrix} 
= M(\xi) \begin{pmatrix}
y(z,x(\xi))\\y^{[1]}(z,x(\xi))\end{pmatrix},
\end{align}
where 
\begin{align}
M(\xi)=\begin{pmatrix} \nu(\xi) & 0\\ (c/4) \nu(\xi)^{-1} Q(\xi) & c \nu(\xi)^{-1} \end{pmatrix},\quad \xi \in [0,1],
\quad {\det}_{\bbC^2}(M(\dott))=c.
\end{align}   

Employing relation \eqref{3.22}, the separated boundary conditions for the function $g(\dott)$ in 
Theorem \ref{t2.2}\,$(i)$ transform into separated boundary  conditions for the transformed function 
$v(\dott)$ as follows,  
\begin{align}
\begin{pmatrix} \cos(\a) & \sin(\a)\\ 0 & 0 \end{pmatrix}M(0)^{-1} \begin{pmatrix}v(0)\\\dot v(0)\end{pmatrix}+ 
\begin{pmatrix} 0 & 0\\ \cos(\b) & -\sin(\b) \end{pmatrix}M(1)^{-1} \begin{pmatrix}v(1)\\ 
\dot v(1)\end{pmatrix},
\end{align}
where $\a,\b\in[0,\pi)$, and the inverse matrix $M^{-1}(\dott)$ has the form
\begin{align}
M(\xi)^{-1} = \begin{pmatrix}  \nu(\xi)^{-1} & 0\\ - (1/4)  \nu(\xi)^{-1}Q(\xi) &  c^{-1} \nu(\xi) \end{pmatrix},\quad \xi \in [0,1],
\end{align}  
or, more explicitly, 
\begin{align}\lb{3.26}
\begin{split}
 c^{-1} \nu(0)\sin(\a)\dot v(0)+ \nu(0)^{-1} \left[\cos(\a)-4^{-1} \sin(\a) Q(0)\right]v(0)&=0,\\
- c^{-1} \nu(1)\sin(\b)\dot v(1)+ \nu(1)^{-1} \left[\cos(\b)+4^{-1} \sin(\b) Q(1)\right]v(1)&=0.
\end{split}
\end{align} 

With the help of relation \eqref{3.22} the coupled boundary conditions for $g(\dott)$ in 
Theorem \ref{t2.2}\,$(ii)$ transform into coupled boundary conditions for $v(\dott)$ via 
\begin{align}
\begin{pmatrix} v(1)\\\dot v(1)\end{pmatrix} 
= e^{i\varphi}\wti{R} \begin{pmatrix} v(0)\\ \dot v(0)\end{pmatrix},\quad \varphi\in[0,\pi),     \lb{3.27}
\end{align} 
where 
\begin{equation} 
\wti{R}=M(1)^{-1} RM(0) \in SL(2,\bbR) 
\end{equation} 
is of the form 
\begin{align}
\begin{split}
\wti{R}_{11}&= \nu(0)^{-1} \nu(1)\left[R_{11}- 4^{-1} Q(0)R_{12}\right],\quad \wti{R}_{12}= c^{-1}   \nu(0)\nu(1)R_{12},\\
\wti{R}_{21}&=c \nu(0)^{-1}  \nu(1)^{-1} \left[R_{21}- 4^{-1} Q(0)R_{22}+ 4^{-1} Q(1)R_{11}- (16)^{-1}Q(0)Q(1)R_{12}\right],\\
\wti{R}_{22}&=\nu(0) \nu(1)^{-1} \left[R_{22}+ 4^{-1} Q(1)R_{12}\right].     \lb{3.29} 
\end{split}
\end{align}

The fundamental system of solutions $\phi(z,\dott,a)$ and $\theta(z,\dott,a)$ of $\tau y=zy$
satisfying \eqref{2.9} 
is transformed into the set of solutions $\Phi(z,\dott ,0)$ and $\Theta(z,\dott ,0)$ of  \eqref{3.17} satisfying 
the conditions
\begin{align}
& \Phi(z,0,0)=0, \qquad \,\,\, \dot \Phi(z,0,0)=c \nu(0)^{-1},      \lb{3.30} \\
&\Theta(z,0,0)=\nu(0), \quad \dot \Theta (z,0,0)= 4^{-1} c \nu(0)^{-1} Q(0),   \lb{3.31}
\end{align} 
where, once again, the derivatives of $\Phi(z,\xi ,0)$ and $\Theta(z,\xi ,0)$ are understood with respect to the variable $\xi $ (cf.\ \eqref{3.17}) and one notes that for fixed $\xi $, each is entire with respect to $z$. By writing a generic solution of \eqref{3.17} as a linear combination of $\Phi(z,\xi ,0)$ and $\Theta(z,\xi ,0)$ and by imposing the 
separated boundary conditions in \eqref{3.26} one obtains the following characteristic function
\begin{align}\lb{3.32}
\begin{split}
\cF_{\a,\b}(z)&= \sin(\a)\big\{c^{-1}  \nu(1)\sin(\b)\dot \Theta (z,1,0)\\
&\quad - \nu(1)^{-1} \left[\cos(\b)+ 4^{-1}\sin(\b) Q(1)\right]\Theta(z,1,0)\big\}\\
&\quad +\cos(\a)\big\{- c^{-1} \nu(1)\sin(\b)\dot \Phi (z,1,0)\\
&\quad + \nu(1)^{-1} \left[\cos(\b)+ 4^{-1} \sin(\b) Q(1)\right]\Phi(z,1,0)\big\},\quad z\in\C.
\end{split}
\end{align}
The zeros of $\cF_{\a,\b}(z)$ represent, including multiplicity, the eigenvalues $\lambda_{A,B,j}$, 
$j \in J$, of the original Sturm--Liouville problem $\tau y=zy$ endowed with the 
separated boundary conditions in \eqref{2.7}.
By repeating this argument for coupled boundary conditions \eqref{2.8} one obtains the characteristic function
\begin{align}\lb{3.33}
\begin{split}
\cF_{\varphi,\wti{R}}(z)&= e^{i\varphi}\big\{2\cos(\varphi)-\big[c^{-1} \nu(0)\wti{R}_{11}
+ 4^{-1} \nu(0)^{-1} Q(0)\wti{R}_{12}\big]\dot \Phi (z,1,0)\\
&\quad +\big[c^{-1} \nu(0)\wti{R}_{21}
+ 4^{-1} \nu(0)^{-1} Q(0)\wti{R}_{22}\big]\Phi(z,1,0)\\
&\quad +\wti{R}_{12} \nu(0)^{-1} \dot \Theta (z,1,0)-\wti{R}_{22} \nu(0)^{-1} \Theta(z,1,0)\big\},\quad z\in\C.
\end{split}
\end{align} 

\begin{remark} \lb{r3.2}
Explicit computations confirm that in the case of separated as well as coupled boundary conditions one finds 
\begin{align}
F_{\a,\b}(z) &= \cF_{\a,\b}(z), \quad z \in \bbC,      \lb{3.34} \\
F_{\varphi,R}(z) &= \cF_{\varphi,\wti R}(z), \quad z \in \bbC.      \lb{3.35} 
\end{align}
${}$ \hfill $\diamond$
\end{remark}

As an example we now consider the case of the Krein--von Neumann extension (see, e.g., \cite{FGKLNS20} and the literature cited therein for details):

\begin{example} \lb{e3.3}
The Krein--von Neumann boundary conditions in terms of the variable $x\in[a,b]$ are characterized by imposing the coupled boundary conditions $\varphi=0$, $R=R_K$ $($cf., e.g., \cite[eq.~(3.35)]{GK19}$)$ with
\begin{align}
R_K=\begin{pmatrix}
\theta(0,b,a) & \phi(0,b,a)\\
\theta^{[1]}(0,b,a) & \phi^{[1]}(0,b,a)
\end{pmatrix}.
\end{align}
In terms of the variable $\xi \in [0,1]$, these boundary conditions are transformed into $\varphi=0$ and $\wti R=\wti R_K$ with
\begin{align}
\wti R_K=\begin{pmatrix}
 \nu(0)^{-1} \big[\Theta(0,1,0)-4^{-1} Q(0)\Phi(0,1,0)\big] & c^{-1}\nu(0)\Phi(0,1,0)\\
 \nu(0)^{-1} \big[\dot \Theta (0,1,0)-4^{-1} Q(0)\dot \Phi (0,1,0)\big] & c^{-1}\nu(0)\dot \Phi (0,1,0)\\
\end{pmatrix}.
\end{align}
By using these parameters in \eqref{3.33} one obtains the transformed characteristic function
\begin{align}
\begin{split}
\cF_{0, \wti R_K} (z)&= 2-c^{-1}\big[\dot \Phi (0,1,0)\Theta(z,1,0)+\Theta(0,1,0)\dot \Phi (z,1,0)     \\
&\hspace*{1.7cm} -\Phi(0,1,0)\dot \Theta (z,1,0)-\dot \Theta (0,1,0)\Phi(z,1,0)\big],\quad z\in\C, 
\end{split}
\end{align}
to be compared with $($see \cite[eq.~(3.36), (3.37)]{GK19}$)$ 
\begin{align}
\begin{split}
F_{0, R_K} (z)&= 2 - \big[\phi^{[1]} (0,b,a)\theta(z,b,a) + \theta(0,b,a) \phi^{[1]} (z,b,a)     \\
&\hspace*{1.15cm} -\phi(0,b,a) \theta^{[1]} (z,b,a) - \theta^{[1]} (0,b,a) \phi(z,b,a)\big], \quad z\in\C.
\end{split}
\end{align}
\end{example}

In order to obtain a large-$z$ asymptotic expansion of the functions \eqref{3.32} and \eqref{3.33}, we need the 
asymptotic expansion of the transformed fundamental set of solutions $\Phi(z,\xi ,0)$ and $\Theta(z,\xi ,0)$. 
To this end, and since the principal results we are focused on in this section are of a local nature with respect 
to $\xi \in [0,1]$, we now envisage that $V(\dott)$ is continued in a sufficiently smooth and compactly supported manner to a function on $\bbR$ (by a slight abuse of notation still abbreviated by $V$),
\begin{equation}
V \in C_0^N(\bbR) \cap C^{\infty}((-\infty,-1)\cup(2,\infty)),      \lb{3.41a}
\end{equation} 
for $N \in \bbN$ to be determined later on. 
In addition, we consider the associated Weyl--Titchmarsh (resp., Jost) solutions $u_{\pm}(z,\dott)$ such that  
for all $x_0 \in \bbR$, 
\begin{equation}
u_+ (z,\dott) \in L^2([x_0,\infty); d\xi), \quad u_-(z,\dott) \in L^2((-\infty,x_0]; d\xi), \quad \Im\big(z^{1/2}\big) > 0.  
\end{equation}
Writing 
\begin{align}\lb{3.39}
u_{\pm}(z,\xi)=\exp\bigg\{\int_{0}^{\xi}dt\ \cS_{\pm}(z,t)\bigg\}, \quad 
\cS_{\pm} (z,\xi) = \f{\dot u_{\pm}(z,\xi)}{u_{\pm}(z,\xi)}, \quad \xi \in \bbR, \; 
\Im\big(z^{1/2}\big) \geq 0
\end{align} 
(the compact support hypothesis on $V$ on $\bbR$, more generally, a suitable short-range, i.e., integrability assumption on $V$, permits the continuous extension of $\cS_{\pm} (z,\dott)$ to $\Im\big(z^{1/2}\big) \geq 0$), one infers that $\cS_{\pm}(z,\dott)$ satisfy the Riccati differential equation
\begin{equation}
\dot S(z,\xi) + S_{\pm}(z,\xi)^2 - V(\xi) + c^2 z = 0, \quad \xi \in \bbR, \; \Im\big(z^{1/2}\big) \geq 0.   \lb{3.39a} 
\end{equation}
In addition, $\cS_{\pm}(z,\xi)$ represent the half-line Weyl--Titchmarsh functions on $[\xi,+\infty)$, respectively, 
$(-\infty,\xi]$, in particular, for each $\xi \in \bbR$, $\pm S_{\pm}(\dott,\xi)$ are Nevanlinna--Herglotz functions on $\bbC_+$ (i.e., analytic on $\bbC_+$ with strictly positive imaginary part on $\bbC_+$).

Inserting the formal asymptotic expansion 
\begin{align}\lb{3.39b}
\cS_{\pm}(z,\dott) \underset{\substack{|z|\to\infty \\ \Im(z^{1/2}) \geq 0}}{=} \pm ic z^{1/2} 
+\sum_{j=1}^{\infty} (\mp1)^jS_{j}(\dott)  z^{-j/2} 
\end{align}
into the Riccati equation \eqref{3.39a} yields the recursion relation 
\begin{align}
\begin{split} 
& S_{1}(\xi)=[i/(2c)] V(\xi),\quad S_{2}(\xi)=[1/4c^2]\dot V(\xi),  \lb{3.44} \\   
& S_{j+1}(\xi)=- [i/(2c)] \bigg[\dot S_{j}(\xi)+\sum_{k=1}^{j-1}S_{k}(\xi)S_{j-k}(\xi)\bigg], \quad j \in \bbN, \; 
\xi \in \bbR.
\end{split} 
\end{align}
The first few terms $S_{j}(\dott)$ explicitly read 
\begin{align}
\begin{split}
S_{3}(\xi)&=\big[i\big/\big(8c^3\big)\big] \big[V^{2}(\xi)-\ddot V(\xi)\big],   \\
S_{4}(\xi)&=- \big[1/16c^4\big]\big[V^{(3)}(\xi)-4V(\xi)\dot V(\xi)\big],    \\
S_{5}(\xi)&= \big[i\big/\big(32c^5\big)\big] \big[2V^{3}(\xi)-5 \dot V(\xi)^{2}-6V(\xi) \ddot V(\xi)+V^{(4)}(\xi)\big], \\
& \text{etc.}
\end{split}
\end{align}
See \cite[Sects.~5, 6]{GHSZ95} for a variety of closely related asymptotic expansions. 

Assuming \eqref{3.41a}, the formal asymptotic expansion \eqref{3.39a} turns into an actual asymptotic expansion of the the type (see \cite{BM97}), 
\begin{align}\lb{3.47A}
\cS_{\pm}(z,\xi) \underset{\substack{|z|\to\infty \\ \Im(z^{1/2}) \geq 0}}{=} \pm ic z^{1/2} 
+\sum_{j=1}^N (\mp1)^jS_{j}(\xi)  z^{-j/2} + \oh\big(|z|^{-N/2}\big), 
\end{align}
with the $\oh\big(|z|^{-N/2}\big)$-term uniform with respect to $\xi \in [0,1]$. 
\begin{remark} \lb{r3.4}
There is an enormous literature available in connection with asymptotic high-energy expansions of Weyl--Titchmarsh $m$-functions (see, e.g., the detailed list in \cite{CG01}) and the associated spectral function, however, much less can be found in connection with (local) uniformity of the error term $\oh\big(|z|^{-N/2}\big)$ with respect to $x$ in expansions of the type \eqref{3.47A}. Notable exceptions are, for instance, \cite{BM97}, \cite{DL91}, \cite{HKS89}, \cite{KK86}, \cite{Ry01}, \cite{Ry02}.  In particular, \cite{BM97} (see \cite[Sects.~1.4, 3.1]{Ma11}) and \cite{DL91} use the theory of transformation operators, while \cite{HKS89} and \cite{KK86} employ a detailed analysis of the Riccati equation \eqref{3.39a}, and \cite{Ry01}, \cite{Ry02} iterate an underlying Volterra integral equation. 
In addition, we note that the compact support hypothesis on $V$ can be relaxed to the condition
\begin{equation}
\int_{\bbR} (1+|x|)dx \, \big|V^{(\ell)}(x)\big| < \infty, \quad 0 \leq \ell \le N. 
\end{equation}
${}$ \hfill $\diamond$
\end{remark}

The correct asymptotic behavior as $|z|\to\infty$ of any solution $u(z,\dott)$ to \eqref{3.17} is given as a linear combination of $u_{\pm}(z,\dott)$,  
\begin{align}\lb{3.46}
u(z,\xi)=\cA(z)u_+(z,\xi)  + \cB(z)u_-(z,\xi), \quad \Im(z) >0, \; \xi \in [0,1], 
\end{align} 
and one notices that the solutions $u_{\pm}(z,\dott)$ satisfy the initial conditions
\begin{align}
u_{\pm}(z,0)=1,\quad \dot u_{\pm} (z,0)=\cS_{\pm}(z,0), \quad \Im(z) > 0.
\end{align} 
Since $W(u_+(z,\dott), u_-(z,\dott))(\xi) \neq 0$, $\xi \in [0,1]$, one infers that
\begin{equation}
\cS^{+}(z,0)-\cS^{-}(z,0) \neq 0, \quad \Im(z) > 0.
\end{equation}  

By imposing the initial conditions \eqref{3.30} and \eqref{3.31} on the function \eqref{3.46}, one obtains an expression for $\Phi(z,\dott ,0)$ and $\Theta(z,\dott,0)$ suitable for an asymptotic expansion. For instance, in the case of $\Phi(z,\xi ,0)$ one obtains 
\begin{align}
\begin{split} 
\Phi(z,\xi ,0)&=\frac{c  \nu(0)^{-1} }{\cS_-(z,0) - \cS_+(z,0)}
\exp\bigg(\int_{0}^{\xi} d \eta \, \cS_-(z,\eta)\bigg)     \\
&\quad\, \times\bigg[1 - \exp\bigg(\int_{0}^{\xi} d \eta \, [\cS_+(z,\eta)-\cS_-(z,\eta)]\bigg)\bigg].
\end{split} 
\end{align}
Furthermore, for large values of $z$, with $\Im(z)>0$, \eqref{3.47A} implies 
\begin{align}
&\exp\bigg(\int_{0}^{\xi} d \eta \, [\cS_+(z,\eta)-\cS_-(z,\eta)]\bigg)    \no \\
&\quad \underset{\substack{|z|\to\infty \\ \Im(z^{1/2}) \geq 0}}{=} \exp\big(2icz^{1/2}\xi\big) 
\exp\bigg(-2\sum_{n=1}^{N}z^{-n+(1/2)}\int_{0}^{\xi} d \eta \,S_{2n-1}(\eta)\bigg)    \lb{3.49} \\
& \hspace*{1.9cm} \times [1+\oh\big(z^{- N+1/2}\big)].  \no
\end{align} 
Since the integrals on the right-hand side of \eqref{3.49} are finite, one finds 
\begin{align}
\exp\bigg(-2\sum_{n=1}^{N}z^{-n+1/2}\int_{0}^{\xi} d \eta \,S_{2n-1}(\eta)\bigg) 
\underset{\substack{|z|\to\infty \\ \Im(z^{1/2}) \geq 0}}{=} \Oh(1), 
\end{align} 
uniformly in $\xi \in [0,1]$. Relations \eqref{3.47A} and \eqref{3.49} permit one to conclude that   
\begin{align}
\exp\bigg(\int_{0}^{\xi} d \eta \, [\cS_+(z,\eta)-\cS_-(z,\eta)]\bigg) \underset{\substack{|z|\to\infty \\ 
\Im(z^{1/2}) \geq 0}}{=} O\big(e^{2icz^{1/2}}\big),
\end{align} 
uniformly for $\xi\in[0,1]$, and therefore,
\begin{align}
\Phi(z,\xi ,0)& \underset{\substack{|z|\to\infty \\ \Im(z^{1/2}) \geq 0}}{=} \frac{c  \nu(0)^{-1} }{\cS_-(z,0)-\cS_+(z,0)}
\exp\bigg(\int_{0}^{\xi} d \eta \, \cS_-(z,\eta)\bigg)\big[1 + \Oh\big(e^{2icz^{1/2} }\big)\big].
\end{align}
Similar arguments permit one to derive the following expressions:
\begin{align}\lb{3.47}
\Theta(z,\xi ,0)& \underset{\substack{|z|\to\infty \\ \Im(z^{1/2}) \geq 0}}{=} \frac{(c/4) \nu(0)^{-1} Q(0)-\nu(0)S_+(z,0)}{\cS_-(z,0)-\cS_+(z,0)}
\exp\bigg(\int_{0}^{\xi} d \eta \, \cS_-(z,\eta)\bigg)     \no \\
& \hspace*{1.6cm} \times\big[1 + \Oh\big(e^{2icz^{1/2} }\big)\big], \\
\dot \Phi (z,\xi ,0)& \underset{\substack{|z|\to\infty \\ \Im(z^{1/2}) \geq 0}}{=} \frac{c  \nu(0)^{-1} \cS_-(z,1)}{\cS_-(z,0)-\cS_+(z,0)}
\exp\bigg(\int_{0}^{\xi} d \eta \, \cS_-(z,\eta)\bigg)\big[1 + \Oh\big(e^{2icz^{1/2} }\big)\big],   \\
\dot \Theta (z,\xi ,0)& \underset{\substack{|z|\to\infty \\ \Im(z^{1/2}) \geq 0}}{=} \frac{\left[(c/4) \nu(0)^{-1} Q(0)-\nu(0)\cS_+(z,0)\right]\cS_-(z,1)}{\cS_-(z,0)-\cS_+(z,0)}      \no \\ 
& \hspace*{1.6cm} \times \exp\bigg(\int_{0}^{\xi} d \eta \, \cS_-(z,\eta)\bigg) 
\big[1 + \Oh\big(e^{2icz^{1/2} }\big)\big], \lb{3.50}
\end{align}
uniformly with respect to $\xi \in [0,1]$.

By utilizing the expressions \eqref{3.47}--\eqref{3.50} in the characteristic functions \eqref{3.32} and \eqref{3.33} we obtain
\begin{align}\lb{3.51}
\cF_{A,B}(z)& \underset{\substack{|z|\to\infty \\ \Im(z^{1/2}) \geq 0}}{=} \frac{1}{\cS^{-}(z,0)-\cS_-(z,0)}
\exp\bigg(\int_{0}^{1} d \eta \, \cS_-(z,\eta)\bigg)   \\
& \hspace*{1.6cm} \times\left[j_{A,B}+k_{A,B}\cS_+(z,0)+\ell_{A,B}\cS_-(z,1)+m_{A,B}\cS_+(z,0)\cS_-(z,1) \right]     \no \\
& \hspace*{1.6cm} \times\big[1 + \Oh\big(e^{2icz^{1/2} }\big)\big].   \no
\end{align} 
The first line on the right-hand side of \eqref{3.51} is entirely independent of boundary conditions, in particular, it does not distinguish between separated and coupled boundary conditions. In contrast, the terms $j_{A,B},k_{A,B},\ell_{A,B}$, and $m_{A,B}$ in the second line on the right-hand side of \eqref{3.51} encode the specific information about the boundary conditions imposed. In the case of separated boundary conditions, where $A, B$ represents 
$\a, \b$ as in \eqref{2.7} one obtains 
\begin{align}
\begin{split}
j_{\a,\b}&=-\dfrac{c}{\nu(0)\nu(1)}\left[\cos(\b)+ (1/4) \sin(\b)Q(1)\right]\left[\cos(\a)- (1/4) \sin(\a)Q(0)\right],\\
k_{\a,\b}&=-\dfrac{\nu(0)}{\nu(1)}\sin(\a)\left[\cos(\b)+ (1/4) \sin(\b)Q(1)\right],\\
\ell_{\a,\b}&=\dfrac{\nu(1)}{\nu(0)}\sin(\b)\left[\cos(\a)- (1/4) \sin(\a)Q(0)\right],\\
m_{\a,\b}&= (1/c)  \nu(0)\nu(1)\sin(\a)\sin(\b).
\end{split}
\end{align}
In the case of coupled boundary conditions, where $A, B$ represents $\varphi, \wti{R}$ as in 
\eqref{3.27}, \eqref{3.29}, one infers 
\begin{align}
\begin{split}
j_{\varphi,\wti{R}}=-e^{i\varphi}\wti{R}_{21},\quad k_{\varphi,\wti{R}}=-e^{i\varphi}\wti{R}_{22},\quad \ell_{\varphi,\wti{R}}=e^{i\varphi}\wti{R}_{11},\quad m_{\varphi,\wti{R}}=e^{i\varphi}\wti{R}_{12}.
\end{split}
\end{align}

For the purpose of the analytic continuation of the spectral $\zeta$-function, one needs the large-$z$ asymptotic expansion of $\ln(\cF_{A,B}(z))$ rather than the one for $\cF_{A,B}(z)$. For this reason we will focus next on the derivation of the large-$z$ asymptotic expansion of the expression 
\begin{align}\lb{3.54}
& \ln(\cF_{A,B}(z)) \underset{\substack{|z|\to\infty \\ \Im(z^{1/2}) \geq 0}}{=} -\ln\big(\cS_+(z,0)-\cS_-(z,0)\big)
+ \int_{0}^{1} d \eta \, \cS_-(z,\eta)     \no \\
& \quad +\ln\big(j_{A,B}+k_{A,B}\cS_+(z,0)+\ell_{A,B}\cS_-(z,1)+m_{A,B}\cS_+(z,0)\cS_-(z,1)\big)     \\
& \quad +\Oh\big(e^{2icz^{1/2} }\big).     \no 
\end{align} 
We can now use the expansion \eqref{3.39a} in \eqref{3.51} to obtain a large-$z$ asymptotic expansion of \eqref{3.54}.
We start with the part of \eqref{3.54} that is independent of the boundary conditions. 
For the integral in \eqref{3.54} one finds
\begin{align}
\int_{0}^{1} d \eta \, \cS_-(z,\eta) \underset{\substack{|z|\to\infty \\ \Im(z^{1/2}) \geq 0}}{=} -iz^{1/2} c+\sum_{m=1}^{N} z^{-m/2}
\int_{0}^{1} d \eta \, S_{m}(\eta) + \oh\big(z^{-N/2}\big).
\end{align}
For the first term in \eqref{3.54} one concludes that 
\begin{align}\lb{3.56}
\cS_+(z,0)-\cS_-(z,0) \underset{\substack{|z|\to\infty \\ \Im(z^{1/2}) \geq 0}}{=} 
2icz^{1/2} \bigg(1+ (i/c) \sum_{j=1}^{N} S_{2j-1}(0) z^{-j}\bigg) + \oh\big(z^{-N+1/2}\big). 
\end{align}
Relation \eqref{3.56} permits one to write
\begin{align}
\ln\big(\cS_+(z,0)-\cS_-(z,0)\big) \underset{\substack{|z|\to\infty \\ \Im(z^{1/2}) \geq 0}}{=} 
\ln(2ic)+2^{-1}\ln(z)+\sum_{m=1}^{N} D_{2m-1} z^{- m} + \oh\big(z^{-N}\big),
\end{align}  
where the terms $D_{2m-1}$ are determined through the formal asymptotic expansion  
\begin{align}
\ln\bigg(1+ (i/c)\sum_{m=1}^{\infty} S_{2m-1}(0) z^{-m}\bigg)  
= \sum_{j=1}^{\infty} D_j z^{- j}.
\end{align}
We refer to \eqref{4.7}--\eqref{4.9} for a recursive formula for $D_j$ in terms of $(i/c)S_{2m-1}(0)$.  
The first few $D_j$ explicitly read
\begin{align}
\begin{split}
D_{1}&= - V(0)\big/\big[2 c^2\big], \quad D_{2}=\big[\ddot V(0)-2 V(0)^2\big]\big/\big[8 c^4\big],\\
D_{3}&= -\big[3 V^{(4)}(0)-24 V(0) \ddot V(0)-15 \dot V(0)^2+16 V(0)^3\big]\big/\big[96 c^6\big],\\
D_{4}&= \big(128 c^8\big)^{-1}\big[V^{(6)}(0)+48 V(0)^2 \ddot V(0)-20 \ddot V(0)^2-12 V(0)V^{(4)}(0)\\
& \quad +60 V(0) \dot V(0)^2-28V^{(3)}(0) \dot V(0)-16 V(0)^4\big],   \\
& \text{etc.} 
\end{split}
\end{align} 

Computing the asymptotic expansion of the last logarithmic term in \eqref{3.54}, namely the term which depends on the boundary conditions, is somewhat more involved. By using the asymptotic expansion \eqref{3.39a} it is not difficult to find
\begin{align} 
\begin{split} 
& j_{A,B}+k_{A,B}\cS^{-}(z,0)+\ell_{A,B}\cS^{+}(z,1)     \\ 
& \quad \underset{\substack{|z|\to\infty \\ \Im(z^{1/2}) \geq 0}}{=} 
-icz^{1/2} (\ell_{A,B}-k_{A,B}) + \sum_{m=0}^{N} \Delta_{m} z^{- m/2} + \oh\big(z^{-N/2}\big), 
\end{split} 
\end{align}
where 
\begin{align}
\Delta_{0}=j_{A,B},\quad \Delta_{m}=\ell_{A,B}S_{m}(1)+(-1)^{m}k_{A,B}S_{m}(0),\quad m \in \bbN,
\end{align}
and
\begin{align}
m_{A,B}\cS^{-}(z,0)\cS^{+}(z,1) \underset{\substack{|z|\to\infty \\ \Im(z^{1/2}) \geq 0}}{=} m_{A,B}c^{2}z
\bigg(1+\sum_{m=2}^{N} \Lambda_{m} z^{- m/2}\bigg) + \oh\big(z^{-(N-2)/2}\big), 
\end{align} 
where 
\begin{align}
\Lambda_{m}=\sum_{\ell=0}^{m}\Omega^{-}_{\ell}(0)\Omega_{m-\ell}^{+}(1), \quad m \in \bbN, \; m \geq 2, 
\end{align} 
with 
\begin{align}
\Omega^{-}_{0}(0)=\Omega_{0}^{+}(1)=1,\quad \Omega_{j}^{+}(x)=(-1)^{j}\Omega_{j}^{-}(x)
= (i/c) S_{j-1}(x),\quad j \in \bbN.
\end{align} 
The first few $\Lambda_{m}$ have the explicit form, 
\begin{align}
\begin{split}
\Lambda_{2}&=- 2^{-1} c^{-2} [V(1)+V(0)], 
\quad \Lambda_{3}=- i 4^{-1} c^{-3} \big[\dot V(1)+\dot V(0)\big],\\
\Lambda_{4}&= 8^{-1} c^{-4} \big[\ddot V(1)+\ddot V(0)-V(0)^2-V(1)^2+2 V(1) V(0)\big],\\
\Lambda_{5}&=i (16)^{-1} c^{-5} \Big[
V^{(3)}(0)-V^{(3)}(1)-2 V(0) \big(2 \dot V(0)+\dot V(1)\big)    \\
& \quad + 2 V(1) \big[\dot V(0)  + 2 \dot V(1)\big]\Big],  \\
& \text{etc.} 
\end{split}
\end{align}
This finally implies 
\begin{align}\lb{3.68}
\begin{split}
& j_{A,B}+k_{A,B}\cS^{-}(z,0)+\ell_{A,B}\cS^{+}(z,1)+m_{A,B}\cS^{-}(z,0)\cS^{+}(z,1)   \\
& \quad \underset{\substack{|z|\to\infty \\ \Im(z^{1/2}) \geq 0}}{=} 
\sum_{m=-2}^{N} \Gamma_{m} z^{- m/2} + \oh\big(z^{-N/2}\big),
\end{split}
\end{align}
where  
\begin{align}
\begin{split} 
& \Gamma_{-2}=m_{A,B}c^{2},\quad \Gamma_{-1}=-ic(\ell_{A,B}-k_{A,B}),    \\
& \Gamma_{m}=\Delta_{m}+m_{A,B}c^{2}\Lambda_{m+2}, \quad m \in \bbN_0.
\end{split} 
\end{align}

Let $\Gamma_{k_0}$ with $k_0 \in\mathbb{Z}$ and $k_0 \geq -2$, be the first non-vanishing term of the series in \eqref{3.68}. Since $\Gamma_{k_0}\neq 0$ one can write
\begin{align}
& \ln\big(j_{A,B}+k_{A,B}\cS^{-}(z,0)+\ell_{A,B}\cS^{+}(z,1)+m_{A,B}\cS^{-}(z,0)\cS^{+}(z,1)\big)    \\
& \quad \underset{\substack{|z|\to\infty \\ \Im(z^{1/2}) \geq 0}}{=} \ln(\Gamma_{k_0})- (k_0/2) \ln(z) 
+ \ln\bigg(1+\sum_{m=1}^{N} [\Gamma_{m+k_0}/\Gamma_{k_0}] z^{- m/2} + \oh\big(z^{-N/2}\big)\bigg),
\no 
\end{align}
which, in turn, yields
\begin{align}
\begin{split}
& \ln\big(j_{A,B}+k_{A,B}\cS^{-}(z,0)+\ell_{A,B}\cS^{+}(z,1)+m_{A,B}\cS^{-}(z,0)\cS^{+}(z,1)\big) \\
& \quad \underset{\substack{|z|\to\infty \\ \Im(z^{1/2}) \geq 0}}{=} 
\ln(\Gamma_{k_0})- (k_0/2) \ln(z)+\sum_{j=1}^{N} \Pi_{j} z^{- j/2} + \oh\big(z^{-N/2}\big),
\end{split}
\end{align}
where the terms $\Pi_{j}$ are obtained via the formal asymptotic expansion 
\begin{align}
\ln\bigg(1+\sum_{m=1}^{\infty} [\Gamma_{m+k_0}/\Gamma_{k_0}] z^{- m/2}\bigg)
= \sum_{j=1}^{\infty} \Pi_{j} z^{- j/2}. 
\end{align} 
Once again we refer to \eqref{4.7}--\eqref{4.9} for a recursive determination 
of $\Pi_j$ in terms of $\Gamma_{m+k_0}/\Gamma_{k_0}$. The first few $\Pi_{m}$ are explicitly of the form, 
\begin{align}
\Pi_{1}&= \Gamma_{1+k_0}/\Gamma_{k_0}, \quad \Pi_{2}=2^{-1} \Gamma^{-2}_{k_0}\big[2\Gamma_{k_0}\Gamma_{k_0 + 2}-\Gamma^2_{k_0 + 1}\big],   \no \\
\Pi_{3}&= 3^{-1} \Gamma^{-3}_{k_0}
\big[\Gamma^3_{k_0 + 1}-3 \Gamma_{k_0}\Gamma_{k_0 + 2}\Gamma_{k_0 + 1}+3\Gamma^2_{k_0}\Gamma_{k_0 + 3}\big],\\
\Pi_{4}&=- 4^{-1} \Gamma^{-4}_{k_0}
\big[\Gamma^4_{k_0 + 1}-4\Gamma_{k_0}\Gamma_{k_0 + 2}\Gamma^2_{k_0 + 1}+4\Gamma^2_{k_0}\Gamma_{k_0 + 3}\Gamma_{k_0 + 1}   \no \\
& \quad +2\Gamma^2_{k_0}\left(\Gamma^2_{k_0+2}-2\Gamma_{k_0}\Gamma_{k_0+4}\right)\big], \no \\
& \text{etc.}  
\end{align}
More explicit expressions for $\Pi_{m}$ in terms of the potential $V$ and its derivatives can be obtained with a simple computer program once the index $k_0$ has been determined. 

Finally, we can provide the large-$z$ asymptotic expansion of the logarithm of the characteristic function in the form 
\begin{align}\lb{3.74}
\begin{split}
\ln(\cF_{A,B}(z))& \underset{\substack{|z|\to\infty \\ \Im(z^{1/2}) \geq 0}}{=}  
-icz^{1/2} -2^{-1} (k_0 + 1) \ln(z)+\ln(\Gamma_{k_0}/(2ic)) \\
& \hspace*{1.6cm} + \sum_{m=1}^{N} \Psi_{m} z^{- m/2} + \oh\big(z^{-N/2}\big), 
\end{split}
\end{align}
where  
\begin{align}
\begin{split}
&\Psi_{2n}=\int_{0}^{1} d \eta \, S_{2n}(\eta)-D_{2n-1}+\Pi_{2n},\quad n\in\mathbb{N},\\
&\Psi_{2n+1}=\int_{0}^{1} d \eta \, S_{2n+1}(\eta)+\Pi_{2n+1},\quad n\in\mathbb{N}_{0}.
\end{split}
\end{align}

\subsection{Analytic Continuation of the Spectral Zeta Function and the Zeta Regularized Functional Determinant} \lb{s3.3}
\hfill

In order to perform the analytic continuation of the spectral $\zeta$-function, we need to investigate the specific behavior for $z\downarrow 0$ and $|z|\to\infty$. The characteristic function $\cF_{A,B}(z)$ is constructed as a linear combination of the basis functions $\phi(z,\dott,a)$ and $\theta(z,\dott,a)$ (or 
equivalently the transformed basis functions $\Phi(z,\dott,0)$ and $\Theta(z,\dott,0)$) and their first quasi-derivatives. We have proved that $\phi(z,\dott,a)$ and $\theta(z,\dott,a)$, and consequently $\Phi(z,\dott,0)$ and $\Theta(z,\dott,0)$, have a small-$z$ asymptotic expansion in the form of a power series in the variable $z$ in Section \ref{s3.1}.
This implies that, in general, the characteristic function $\cF_{A,B}(z)$ has a small-$z$ asymptotic expansion of the form
\begin{align}\lb{3.76}
\cF_{A,B}(z)=\cF_{m_0} z^{m_0}+\sum_{m=m_0 +1}^{\infty}\cF_{m}z^{m},
\end{align}  
where $m_0 \in\{0,1,2\}$ represents the multiplicity of the zero eigenvalue and $\cF_{m_0}\neq 0$. The asymptotic expansion \eqref{3.76} suggests that the appropriate characteristic function to use 
in the integral representation of the spectral $\zeta$-function is $z^{-{m_0}}\cF_{A,B}(z)$ rather than simply $\cF_{A,B}(z)$ (obviously the two coincide when no zero eigenvalue is present). In this case it is easy to verify that 
\begin{align}
\frac{d}{dz}\ln\left(\cF_{A,B}(z) z^{-m_0}\right)\underset{|z| \downarrow 0}{=} \Oh(1).   \lb{3.85a}
\end{align}

From the large-$z$ asymptotic expansion \eqref{3.74} of the characteristic function, namely,
\begin{align}
\begin{split} 
\ln(\cF_{A,B}(z)) &\underset{\substack{|z|\to\infty \\ \Im(z^{1/2}) \geq 0}}{=} 
-icz^{1/2} - [(k_0 + 1)/2] \ln(z) + \ln(\Gamma_{k_0}/(2ic))     \\
& \hspace*{1.6cm} + \sum_{m=1}^N \Psi_{m} z^{- m/2} + \oh\big(|z|^{-N/2}\big),  
\end{split} 
\end{align}   
one readily infers that   
\begin{align}
\frac{d}{dz}\ln(\cF_{A,B}(z) z^{- m_0}) \underset{\substack{|z|\to\infty \\ \Im(z^{1/2}) \geq 0}}{=} 
\Oh\big(|z|^{-1/2}\big).   \lb{3.87}
\end{align}

The asymptotic behaviors in \eqref{3.85a} and \eqref{3.87} justify deforming the contour $\gamma$ in the integral representation \eqref{2.40} to one that surrounds the branch cut $R_{\psi}$ as shown in Figure \ref{f2}. This contour deformation leads to the following integral representation (with $\psi$ introduced in \eqref{2.40a}) 
\begin{align}\lb{3.80}
\zeta(s;T_{A,B})=e^{is(\pi-\psi)}\pi^{-1} \sin(\pi s) \int_{0}^{\infty} dt \, t^{-s}
\frac{d}{dt}\ln\left(\cF_{A,B}(te^{i\psi}) t^{- m_0}e^{- i m_0 \psi}\right),
\end{align}
which is valid in the region $1/2<\Re(s)<1$. To obtain the analytic continuation of \eqref{3.80} to the left of the abscissa of convergence $\Re(s)=1/2$ we subtract and then add
$N$ terms of the large-$z$ asymptotic expansion of $\ln\left(\cF_{A,B}(te^{i\psi})
t^{- m_0} e^{-i m_0 \psi}\right)$. This process leads to the 
following expression of the spectral $\zeta$-function
\begin{align}\lb{3.81}
\zeta(s;T_{A,B})=Z(s,A,B)+\sum_{j=-1}^{N}h_{j}(s,A,B),
\end{align}
which is valid in the region $-(N+1)/2<\Re(s)<1$. The explicit form of the functions in the analytically continued expression of $\zeta(s;T_{A,B})$ in \eqref{3.81} is 
\begin{align}\lb{3.82}
Z(s,A,B)& =e^{is(\pi-\psi)}\pi^{-1} \sin(\pi s) 
\int_{0}^{\infty} dt \, t^{-s}\frac{d}{dt}\bigg\{\ln\left(\cF_{A,B}(te^{i\psi})t^{-m_0}e^{- i m_0 \psi}\right)
\no \\
& \quad -H(t-1)\bigg[-ic t^{1/2}e^{i \psi/2} - [((k_0 + 1)/2)+ m_0] \ln (t)    \\
& \quad - \big[((k_0 + 1)/2)+ m_0\big] i \psi + \ln(\Gamma_{k_0}/(2ic))  
+\sum_{n=1}^{N}\Psi_{n} e^{-in \psi/2} t^{- n/2}\bigg]\bigg\},     \no 
\end{align}
where $H(s) = \begin{cases} 1, & s > 0, \\
0, & s < 0, \end{cases}$  represents the Heaviside function, and 
\begin{align}\lb{3.83}
\begin{split}
h_{-1}(s,A,B)&=-ie^{is(\pi-\psi)}\pi^{-1} \sin(\pi s) c\,e^{i \psi/2}/(2s-1),   \\
h_{0}(s,A,B)&=- (k_0 + 1 + 2 m_0) e^{is(\pi-\psi)} (2 \pi s)^{-1} \sin(\pi s),  \\
h_{n}(s,A,B)&=-e^{is(\pi-\psi)} \pi^{-1} \sin(\pi s) [n/(2s+n)] e^{-in \psi/2} \Psi_{n},\quad n \in \bbN.
\end{split}
\end{align} 

Given the expression \eqref{3.81} we are now able to compute the zeta regularized functional determinant in terms of $\zeta'(0;T_{A,B})$ as in \cite[Thm. 2.9]{GK19}.
For the purpose of computing $\zeta'(0;T_{A,B})$, it is sufficient to set $N=0$ in \eqref{3.81} to obtain
\begin{align}
\zeta'(0;T_{A,B})=Z'(0,A,B)+h'_{-1}(0,A,B)+h'_{0}(0,A,B).
\end{align} 

By computing the derivative with respect to $s$ of \eqref{3.82} and the first two expressions in \eqref{3.83} at $s=0$ one obtains the remarkably simple formula
\begin{align}\lb{3.85}
\zeta'(0;T_{A,B})=i\pi n-\ln(2c|\cF_{m_0}/\Gamma_{k_0}|),
\end{align}
where $n$ is the number of strictly negative eigenvalues of $T_{A,B}$.

\section{Computing Spectral Zeta Function Values and Traces for Regular Sturm--Liouville Operators} \lb{s4}

We have now completed the necessary preparations to give the main theorem for computing values of the spectral $\zeta$-function for self-adjoint regular Sturm--Liouville operators when imposing either separated or coupled boundary conditions. When zero is not an eigenvalue we also find an expression for computing the trace of the inverse Sturm--Liouville operator.

\begin{theorem} \lb{t4.1}
Assume Hypothesis \ref{h2.1}, denote by $T_{A,B}$ the self-adjoint extension of $T_{min}$ with either separated or coupled boundary conditions as described in Theorem \ref{t2.2}, and 
let $m_0 = 0,1,2$, denote the multiplicity of zero as an eigenvalue of $T_{A,B}$ $($with $m_0 = 0$ denoting zero is not an eigenvalue$)$. Suppose that $F_{A,B}(z)$ given in \eqref{2.40} has the series expansion, 
\begin{align}\lb{4.1}
    F_{A,B}(z)=\sum_{j=0}^\infty a_jz^j,\quad 0\leq |z|\text{ sufficiently small}.
\end{align}
Then,
\begin{align}\lb{4.2}
    \zeta(n;T_{A,B})=-\Res \left[ z^{-n}\dfrac{d}{dz}\ln (F_{A,B}(z));\ z=0 \right]=-n \, b_n, \quad n \in \bbN, 
\end{align}
where
\begin{align}\lb{4.3}
    \begin{split}
        &b_1= a_{1+m_0}/a_{m_0},\\
        &b_j= [a_{j+m_0}/a_{m_0}]-\sum_{\ell=1}^{j-1} [\ell/j]
        [a_{j-\ell+m_0}/a_{m_0}] b_{\ell},\quad j \in \bbN, \; j\geq 2.
    \end{split}
\end{align}
In particular, if zero is not an eigenvalue of $T_{A,B}$, then
\begin{align}\lb{4.4}
    \text{\rm tr}_{L^2_r((a,b))}\big(T_{A,B}^{-1}\big)=\zeta(1;T_{A,B})=- a_{1}/a_{0}.
\end{align}
\end{theorem}
\begin{proof} 
The residue in equation \eqref{4.2} coincides with the $z^{-1}$ coefficient of the Laurent expansion, in the neighborhood of $z=0$, of the integrand in
\eqref{2.42}. By using the expansion \eqref{4.1} one obtains, for $|z|\geq 0$ sufficiently small and for $n\in\N$, that
\begin{align}
z^{-n}\dfrac{d}{dz}\ln (F_{A,B}(z))=z^{-n}\dfrac{d}{dz}\ln\bigg(\sum_{j=0}^\infty a_j z^j\bigg).
\end{align}
Since $z=0$ can be an eigenvalue of multiplicity at most 2, the expansion can be rewritten as follows, 
\begin{align}
\no  z^{-n}\dfrac{d}{dz}\ln (F_{A,B}(z))&=z^{-n}\dfrac{d}{dz}\ln\bigg(\sum_{j=m_0}^\infty a_j z^j\bigg)\\
&\no =z^{-n}\dfrac{d}{dz}\bigg(\ln\big(a_{m_0}  z^{m_0}\big) 
+ \ln\bigg(1+\sum_{j=1}^\infty [a_{j+m_0}/a_{m_0}] z^j\bigg)\bigg)\\
&= m_0 z^{-n-1}+z^{-n}\dfrac{d}{dz}\ln\bigg(1+\sum_{j=1}^\infty [a_{j+m_0}/a_{m_0}] z^j\bigg).
\end{align}
Since $n\in\N$, the term $m_0 z^{-n-1}$ never contributes to the residue and the only contribution comes from the $z^n$ coefficient of the small-$|z|$ asymptotic expansion of the logarithm on the right-hand side. This expansion can be obtained by making use of the fact that if $F$ has the analytic expansion
\begin{align}\lb{4.7}
    F(z)=\sum_{m=1}^\infty c_m z^m,\quad 0\leq |z|\text{ sufficiently small},
\end{align}
then
\begin{align}\lb{4.8}
    \ln(1+F(z)) = \sum_{m=1}^\infty d_m z^m, \quad 0\leq |z|\text{ sufficiently small},
\end{align}
where
\begin{equation}\lb{4.9}
d_1=c_1, \quad d_j=c_j-\sum_{\ell=1}^{j-1} [\ell /j] c_{j-\ell}d_{\ell},\quad j \in \bbN, \; j\geq 2.
\end{equation}
By using \eqref{4.8} one obtains
\begin{align}
\ln\bigg(1+\sum_{j=1}^\infty [a_{j+m_0}/a_{m_0}] z^j\bigg)=\sum_{j=1}^\infty b_j z^j,
\end{align}
with the coefficients $b_j$ given by equation \eqref{4.3}. From the last expansion one finally obtains
\begin{align}
z^{-n}\dfrac{d}{dz}\ln (F_{A,B}(z))=z^{-n}\dfrac{d}{dz}\ln \bigg(\sum_{j=1}^\infty a_j z^j\bigg)=\sum_{j=1}^\infty j b_j z^{j-n-1}.
\end{align}
This is the Laurent expansion, and from it one can deduce that 
\begin{align}
\Res \left[ z^{-n}\dfrac{d}{dz}\ln (F_{A,B}(z));\ z=0 \right]=n \, b_n,\quad n\in\N,
\end{align}
proving \eqref{4.2}.

Assertion \eqref{4.4} about the trace of the inverse operator when $z=0$ is not an eigenvalue is obtained by noting
\begin{align}
-\dfrac{d}{dz}\ln(F_{A,B}(z))\bigg|_{z=0}=-d_1=- a_1/a_0 
\end{align}
from the analytic expansions \eqref{4.7} and \eqref{4.8}, and upon applying Theorem \ref{t2.4}.
\end{proof}

This theorem allows one to utilize the series expansions found in the previous section in order to express the $\zeta$-function values for each of the boundary conditions considered.

\subsection{Computing Spectral Zeta Function Values and Traces for Separated Boundary Conditions} \lb{s4.1}
\hfill

We begin by applying Theorem \ref{t4.1} to find an expression for values of $\zeta(n;T_{\a,\b})$ when imposing separated boundary conditions.

\begin{theorem}\lb{t4.2}
Assume Hypothesis \ref{h2.1}, consider $T_{\a,\b}$ as described in Theorem \ref{t2.2}\,$(i)$, and let $m_0=0,1$, denote the multiplicity of zero as an eigenvalue of $T_{\a,\b}$. Then, 
\begin{align}\lb{4.14}
    \zeta(n;T_{\a,\b})=-\Res \left[ z^{-n}\dfrac{d}{dz}\ln (F_{\a,\b}(z));\ z=0 \right]=-n \, b_n, \quad n \in \bbN, 
\end{align}
where
\begin{align}
&\no \resizebox{\textwidth}{!}{$b_1=\dfrac{\cos(\a)\big[\cos(\b) \phi_{1+m_0}(b)-\sin(\b) \phi^{[1]}_{1+m_0}(b)\big]-\sin(\a)\big[\cos(\b) \theta_{1+m_0}(b)-\sin(\b) \theta^{[1]}_{1+m_0}(b)\big]}{\cos(\a)\big[\cos(\b) \phi_{m_0}(b)-\sin(\b) \phi^{[1]}_{m_0}(b)\big]-\sin(\a)\big[\cos(\b) \theta_{m_0}(b)-\sin(\b) \theta^{[1]}_{m_0}(b)\big]},$}\\
\no \\
&\no \resizebox{\textwidth}{!}{$b_j=\dfrac{\cos(\a)\big[\cos(\b) \phi_{j+m_0}(b)-\sin(\b) \phi^{[1]}_{j+m_0}(b)\big]-\sin(\a)\big[\cos(\b) \theta_{j+m_0}(b)-\sin(\b) \theta^{[1]}_{j+m_0}(b)\big]}{\cos(\a)\big[\cos(\b) \phi_{m_0}(b)-\sin(\b) \phi^{[1]}_{m_0}(b)\big]-\sin(\a)\big[\cos(\b) \theta_{m_0}(b)-\sin(\b) \theta^{[1]}_{m_0}(b)\big]}$}\\
&\no \resizebox{\textwidth}{!}{$\quad\quad\ \ -\displaystyle\sum_{\ell=1}^{j-1}\left(\dfrac{\ell}{j}\right)\dfrac{\cos(\a)\big[\cos(\b) \phi_{j-\ell+m_0}(b)-\sin(\b) \phi^{[1]}_{j-\ell+m_0}(b)\big]-\sin(\a)\big[\cos(\b) \theta_{j-\ell+m_0}(b)-\sin(\b) \theta^{[1]}_{j-\ell+m_0}(b)\big]}{\cos(\a)\big[\cos(\b) \phi_{m_0}(b)-\sin(\b) \phi^{[1]}_{m_0}(b)\big]-\sin(\a)\big[\cos(\b) \theta_{m_0}(b)-\sin(\b) \theta^{[1]}_{m_0}(b)\big]}b_\ell,$}\\
&\hspace{9cm} j \in \bbN, \; j\geq 2.\lb{4.15}
\end{align}
In particular, if zero is not an eigenvalue of $T_{\a,\b}$, then
\begin{align}
    &\text{\rm tr}_{L^2_r((a,b))}\big(T^{-1}_{\a,\b}\big)=\zeta(1;T_{\a,\b})\\
    &\resizebox{\textwidth}{!}{$\quad=-\dfrac{\cos(\a)\big[\cos(\b)\ \phi_{1}(b)-\sin(\b)\ \phi^{[1]}_{1}(b)\big]-\sin(\a)\big[\cos(\b)\ \theta_{1}(b)-\sin(\b)\ \theta^{[1]}_{1}(b)\big]}{\cos(\a)\big[\cos(\b)\ \phi_{0}(b)-\sin(\b)\ \phi^{[1]}_{0}(b)\big]-\sin(\a)\big[\cos(\b)\ \theta_{0}(b)-\sin(\b)\ \theta^{[1]}_{0}(b)\big]}.$}\no 
\end{align}
\end{theorem}
\begin{proof}
One substitutes \eqref{3.4}, \eqref{3.6}, \eqref{3.8}, and \eqref{3.11} into equation \eqref{2.16} for $\a,\b\in [0,\pi)$ to find
\begin{align}\lb{4.17}
\begin{split} 
F_{\a,\b}(z)=\sum_{m=0}^\infty \big\{&\cos(\a)\big[\cos(\b)\ \phi_m(b)-\sin(\b)\ \phi^{[1]}_m(b)\big]\\
&-\sin(\a)\big[\cos(\b)\ \theta_m(b)-\sin(\b)\ \theta^{[1]}_m(b)\big]\big\}z^m.
\end{split}
\end{align}
From \eqref{4.17} one proves
the assertion by applying Theorem \ref{t4.1} with
\begin{align}
\begin{split}
a_k&=\cos(\a)\big[\cos(\b)\ \phi_k(b)-\sin(\b)\ \phi^{[1]}_k(b)\big]\\
& \quad -\sin(\a)\big[-\sin(\b)\ \theta^{[1]}_k(b)+\cos(\b)\ \theta_k(b)\big],\quad k\in\N.
\end{split}
\end{align}
\end{proof}

We now give a few corollaries that will be of use in the context of specific boundary conditions. One notes that for Dirichlet boundary conditions one has $\a=\b=0$ and for Neumann boundary conditions one has $\a=\b=\pi/2$.

\begin{corollary}[Dirichlet boundary conditions] \lb{c4.3}
Assume Hypothesis \ref{h2.1}, consider $T_{0,0}$ as described in case Theorem \ref{t2.2}\,$(i)$, and let $m_0=0,1$, denote the multiplicity of zero as an eigenvalue of $T_{0,0}$. Then,
\begin{align}\lb{4.19}
    \zeta(n;T_{0,0})=-\Res \left[ z^{-n}\dfrac{d}{dz}\ln (F_{0,0}(z));\ z=0 \right]=-n \, b_n, \quad n \in \bbN, 
\end{align}
where
\begin{align}\lb{4.20}
    \begin{split}
        &b_1= \phi_{1+m_0}(b)/\phi_{m_0}(b),\\
        &b_j= [\phi_{j+m_0}(b)/\phi_{m_0}(b)] 
        - \sum_{\ell=1}^{j-1} [\ell/j] [\phi_{j-\ell+m_0}(b)/\phi_{m_0}(b)] b_{\ell}, 
        \quad j \in \bbN, \; \geq 2.
    \end{split}
\end{align}
In particular, if zero is not an eigenvalue of $T_{0,0}$, then
\begin{align}
    \text{\rm tr}_{L^2_r((a,b))}\big(T^{-1}_{0,0}\big)=\zeta(1;T_{0,0})=- \phi_1(b)/\phi_0(b).
\end{align}
\end{corollary}
\begin{proof}
Take $\a=\b=0$ in Theorem \ref{t4.2}.
\end{proof}

In particular, one finds explicitly for $n=2,3,4$, when zero is not an eigenvalue of $T_{0,0}$:
\begin{align}
\no  &\zeta(2;T_{0,0})=-2b_2=-2\left[\dfrac{\phi_2(b)}{\phi_0(b)}-\dfrac{[\phi_1(b)]^2}{2[\phi_0(b)]^2}\right],\\
&\zeta(3;T_{0,0})=-3b_3=-3\Bigg[\dfrac{\phi_3(b)}{\phi_0(b)}-\dfrac{\phi_1(b)\phi_2(b)}{[\phi_0(b)]^2}+\dfrac{[\phi_1(b)]^3}{3[\phi_0(b)]^3}\Bigg],\\
\begin{split}
\no  &\resizebox{\textwidth}{!}{$\zeta(4;T_{0,0})=-4b_4=-4\Bigg[\dfrac{\phi_4(b)}{\phi_0(b)}-\dfrac{\phi_1(b)\phi_3(b)}{[\phi_0(b)]^2}-\dfrac{[\phi_2(b)]^2}{2[\phi_0(b)]^2}+\dfrac{[\phi_1(b)]^2\phi_2(b)}{[\phi_0(b)]^3}-\dfrac{[\phi_1(b)]^4}{4[\phi_0(b)]^4}\Bigg].$}
\end{split}
\end{align}

One also finds explicitly for $n=2,3,4$, when zero is a simple eigenvalue of $T_{0,0}$:
\begin{align}
\no  &\zeta(2;T_{0,0})=-2b_2=-2\left[\dfrac{\phi_3(b)}{\phi_1(b)}-\dfrac{[\phi_2(b)]^2}{2[\phi_1(b)]^2}\right],\\
&\zeta(3;T_{0,0})=-3b_3=-3\Bigg[\dfrac{\phi_4(b)}{\phi_1(b)}-\dfrac{\phi_2(b)\phi_3(b)}{[\phi_1(b)]^2}+\dfrac{[\phi_2(b)]^3}{3[\phi_1(b)]^3}\Bigg],\\
\begin{split}
\no &\resizebox{\textwidth}{!}{$\zeta(4;T_{0,0})=-4b_4=-4\Bigg[\dfrac{\phi_5(b)}{\phi_1(b)}-\dfrac{\phi_2(b)\phi_4(b)}{[\phi_1(b)]^2}-\dfrac{[\phi_3(b)]^2}{2[\phi_1(b)]^2}+\dfrac{[\phi_2(b)]^2\phi_3(b)}{[\phi_1(b)]^3}-\dfrac{[\phi_2(b)]^4}{4[\phi_1(b)]^4}\Bigg].$}
\end{split}
\end{align}

\begin{corollary}[Dirichlet boundary condition at $a$]\lb{c4.4}
Assume Hypothesis \ref{h2.1}, consider $T_{0,\b}$ as described in Theorem \ref{t2.2}\,$(i)$, and 
let $m_0=0,1$, denote the multiplicity of zero as an eigenvalue of $T_{0,\b}$. Then,
\begin{align}\lb{4.24}
    \zeta(n;T_{0,\b})=-\Res \left[ z^{-n}\dfrac{d}{dz}\ln (F_{0,\b}(z));\ z=0 \right]=-n \, b_n, \quad n \in \bbN, 
\end{align}
where
\begin{align}\lb{4.25}
    \begin{split}
        b_1&= \dfrac{\cos(\b)\phi_{1+m_0}(b)-\sin(\b)\phi^{[1]}_{1+m_0}(b)}{\cos(\b)\phi_{m_0}(b)-\sin(\b)\phi^{[1]}_{m_0}(b)},\\
        b_j&= \dfrac{\cos(\b)\phi_{j+m_0}(b)-\sin(\b)\phi^{[1]}_{j+m_0}(b)}{\cos(\b)\phi_{m_0}(b)-\sin(\b)\phi^{[1]}_{m_0}(b)}\\
        & \quad -\sum_{\ell=1}^{j-1} [\ell/j] 
        \dfrac{\cos(\b)\phi_{j-\ell+m_0}(b)-\sin(\b)\phi^{[1]}_{j-\ell+m_0}(b)}{\cos(\b)\phi_{m_0}(b)-\sin(\b)\phi^{[1]}_{m_0}(b)}b_{\ell},\quad j \in \bbN, \; j\geq 2.
    \end{split}
\end{align}
In particular, if zero is not an eigenvalue of $T_{0,\b}$, then
\begin{align}
    \text{\rm tr}_{L^2_r((a,b))}\big(T^{-1}_{0,\b}\big)=\zeta(1;T_{0,\b})=-\dfrac{\cos(\b)\phi_{1}(b)-\sin(\b)\phi^{[1]}_{1}(b)}{\cos(\b)\phi_{0}(b)-\sin(\b)\phi^{[1]}_{0}(b)}.
\end{align}
\end{corollary}
\begin{proof}
Take $\a=0$ in Theorem \ref{t4.2}.
\end{proof}

\begin{corollary}[Dirichlet boundary condition at $b$]\lb{c4.5}
Assume Hypothesis \ref{h2.1}, consider $T_{\a,0}$ as described in Theorem \ref{t2.2}\,$(i)$, and let 
$m_0=0,1$, denote the multiplicity of zero as an eigenvalue of $T_{\a,0}$. Then, 
\begin{align}\lb{4.27}
    \zeta(n;T_{\a,0})=-\Res \left[ z^{-n}\dfrac{d}{dz}\ln (F_{\a,0}(z));\ z=0 \right]=-n \, b_n, \quad n \in \bbN, 
\end{align}
where
\begin{align}\lb{4.28}
    \begin{split}
        b_1&= \dfrac{\cos(\a)\phi_{1+m_0}(b)-\sin(\a)\theta_{1+m_0}(b)}{\cos(\a)\phi_{m_0}(b)-\sin(\a)\theta_{m_0}(b)},\\
        b_j&= \dfrac{\cos(\a)\phi_{j+m_0}(b)-\sin(\a)\theta_{j+m_0}(b)}{\cos(\a)\phi_{m_0}(b)-\sin(\a)\theta_{m_0}(b)}\\
        & \quad -\sum_{\ell=1}^{j-1} [\ell/j]
        \dfrac{\cos(\a)\phi_{j-\ell+m_0}(b)-\sin(\a)\theta_{j-\ell+m_0}(b)}{\cos(\a)\phi_{m_0}(b)-\sin(\a)\theta_{m_0}(b)}b_{\ell},\quad j \in \bbN, \; j\geq 2.
    \end{split}
\end{align}
In particular, if zero is not an eigenvalue of $T_{\a,0}$, then
\begin{align}
    \text{\rm tr}_{L^2_r((a,b))}\big(T^{-1}_{\a,0}\big)=\zeta(1;T_{\a,0})=-\dfrac{\cos(\a)\phi_{1}(b)-\sin(\a)\theta_{1}(b)}{\cos(\a)\phi_{0}(b)-\sin(\a)\theta_{0}(b)}.
\end{align}
\end{corollary}
\begin{proof}
Take $\b=0$ in Theorem \ref{t4.2}.
\end{proof}

\begin{corollary}[Neumann boundary conditions]\lb{c4.6}
Assume Hypothesis \ref{h2.1}, consider $T_{\pi/2,\pi/2}$ as described in Theorem \ref{t2.2}\,$(i)$, and let $m_0=0,1$, denote the multiplicity of zero as an eigenvalue of $T_{\pi/2,\pi/2}$. Then, 
\begin{align}\lb{4.30}
    \zeta(n;T_{\pi/2,\pi/2})=-\Res \left[ z^{-n}\dfrac{d}{dz}\ln (F_{\pi/2,\pi/2}(z));\ z=0 \right]=-n \, b_n, \quad n \in \bbN, 
\end{align}
where
\begin{align}
&b_1= \theta^{[1]}_{1+m_0}(b)\big/\theta^{[1]}_{m_0}(b),\\
&b_j=\theta^{[1]}_{j+m_0}\big/(b)\theta^{[1]}_{m_0}(b) 
- \sum_{\ell=1}^{j-1} [\ell/j] \big[\theta^{[1]}_{j-\ell+m_0}(b)\big/\theta^{[1]}_{m_0}(b)\big] b_\ell, 
\quad j \in \bbN, \; j\geq 2.  \no 
\end{align}
In particular, if zero is not an eigenvalue of $T_{\pi/2,\pi/2}$, then
\begin{align}
    \text{\rm tr}_{L^2_r((a,b))}\big(T^{-1}_{\pi/2,\pi/2}\big)=\zeta(1;T_{\pi/2,\pi/2})=-\theta^{[1]}_{1}(b)\big/\theta^{[1]}_{0}(b).
\end{align}
\end{corollary}
\begin{proof}
Take $\a=\b=\pi/2$ in Theorem \ref{t4.2}.
\end{proof}

These are only a few of the most considered separated boundary conditions that have been singled out. One can also consider Neumann boundary conditions at only one endpoint, or any other combination of separated boundary conditions, by referring back to Theorem \ref{t4.2} with the appropriate values chosen for $\a,\b\in[0,\pi)$.

\subsection{Computing Spectral Zeta Function Values and Traces for Coupled Boundary Conditions} \lb{s4.2}
\hfill

We now apply Theorem \ref{t4.1} to find values of $\zeta(n;T_{\varphi,R})$ when imposing coupled boundary conditions. Notice that according to \cite{FGKLNS20}, zero is an eigenvalue of multiplicity 2 only for the Krein--von Neumann extension.

\begin{theorem}\lb{t4.7}
Assume Hypothesis \ref{h2.1}, consider $T_{\varphi,R}$ as described in Theorem \ref{t2.2}\,$(ii)$, and let $m_0=0,1$, denote the multiplicity of zero as an eigenvalue of $T_{\varphi,R}$. Then,
\begin{align}\lb{4.33}
    \zeta(n;T_{\varphi,R})=-\Res \left[ z^{-n}\dfrac{d}{dz}\ln (F_{\varphi,R}(z));\ z=0 \right]=-n \, b_n, 
    \quad n \in \bbN, 
\end{align}
where for $m_0=0$,
\begin{align}
        &b_1=\dfrac{e^{i\varphi}\big(R_{12}\theta_1^{[1]}(b)-R_{22}\theta_1(b)+R_{21}\phi_1(b)-R_{11}\phi_1^{[1]}(b)\big)}{e^{i\varphi}\big(R_{12}\theta_0^{[1]}(b)-R_{22}\theta_0(b)+R_{21}\phi_0(b)-R_{11}\phi_0^{[1]}(b)\big)+e^{2i\varphi}+1},    \no \\
        &b_j=\dfrac{e^{i\varphi}\big(R_{12}\theta_{j}^{[1]}(b)-R_{22}\theta_{j}(b)+R_{21}\phi_{j}(b)-R_{11}\phi_{j}^{[1]}(b)\big)}{e^{i\varphi}\big(R_{12}\theta_0^{[1]}(b)-R_{22}\theta_0(b)+R_{21}\phi_0(b)-R_{11}\phi_0^{[1]}(b)\big)+e^{2i\varphi}+1}  
        \lb{4.34} \\
        & \qquad -\displaystyle\sum_{\ell=1}^{j-1} \f{\ell}{j}    
  \dfrac{e^{i\varphi}\big(R_{12}\theta_{j-\ell}^{[1]}(b)-R_{22}\theta_{j-\ell}(b)+R_{21}\phi_{j-\ell}(b)-R_{11}\phi_{j-\ell}^{[1]}(b)\big)}{e^{i\varphi}\big(R_{12}\theta_0^{[1]}(b)-R_{22}\theta_0(b)+R_{21}\phi_0(b)-R_{11}\phi_0^{[1]}(b)\big)+e^{2i\varphi}+1}b_{\ell},     \no \\
        &\hspace{9.8cm} j \in \bbN, \; j\geq 2,   \no
\end{align}
and for $m_0=1$,
\begin{align}
        &b_1=\dfrac{e^{i\varphi}\big(R_{12}\theta_{2}^{[1]}(b)-R_{22}\theta_{2}(b)+R_{21}\phi_{2}(b)-R_{11}\phi_{2}^{[1]}(b)\big)}{e^{i\varphi}\big(R_{12}\theta_1^{[1]}(b)-R_{22}\theta_1(b)+R_{21}\phi_1(b)-R_{11}\phi_1^{[1]}(b)\big)},     \no \\
        &b_j=\dfrac{e^{i\varphi}\big(R_{12}\theta_{j+1}^{[1]}(b)-R_{22}\theta_{j+1}(b)+R_{21}\phi_{j+1}(b)-R_{11}\phi_{j+1}^{[1]}(b)\big)}{e^{i\varphi}\big(R_{12}\theta_1^{[1]}(b)-R_{22}\theta_1(b)+R_{21}\phi_1(b)-R_{11}\phi_1^{[1]}(b)\big)}   
        \lb{4.35} \\
        & \qquad -\displaystyle\sum_{\ell=1}^{j-1} \f{\ell}{j}  
        \dfrac{e^{i\varphi}\big(R_{12}\theta_{j-\ell+1}^{[1]}(b)-R_{22}\theta_{j-\ell+1}(b)+R_{21}\phi_{j-\ell+1}(b)-R_{11}\phi_{j-\ell+1}^{[1]}(b)\big)}{e^{i\varphi}\big(R_{12}\theta_1^{[1]}(b)-R_{22}\theta_1(b)+R_{21}\phi_1(b)-R_{11}\phi_1^{[1]}(b)\big)} b_{\ell},     \no \\
        &\hspace{10.55cm} j \in \bbN, \; j\geq 2.    \no 
\end{align}
In particular, if zero is not an eigenvalue of $T_{\varphi,R}$, then
\begin{align}
\begin{split}
& \text{\rm tr}_{L^2_r((a,b))}\big(T_{\varphi,R}^{-1}\big) = \zeta(1;T_{\varphi,R})\\
& \quad 
=\dfrac{-e^{i\varphi}\big(R_{12}\theta_1^{[1]}(b)-R_{22}\theta_1(b)+R_{21}\phi_1(b)-R_{11}\phi_1^{[1]}(b)\big)}{e^{i\varphi}\big(R_{12}\theta_0^{[1]}(b)-R_{22}\theta_0(b)+R_{21}\phi_0(b)-R_{11}\phi_0^{[1]}(b)\big)+e^{2i\varphi}+1}.
\end{split}
\end{align}
\end{theorem}
\begin{proof}
Substituting \eqref{3.4}, \eqref{3.6}, \eqref{3.8}, and \eqref{3.11} into equation \eqref{2.17} yields
\begin{align}
F_{\varphi,R}(0)=e^{i\varphi}\big(R_{12}\theta_0^{[1]}(b)-R_{22}\theta_0(b)+R_{21}\phi_0(b)-R_{11}\phi_0^{[1]}(b)\big)+e^{2i\varphi}+1.
\end{align}
Thus, the coefficient of the $z^m$ term for $m\geq1$ in the series is given by
\begin{align}
e^{i\varphi}\big(R_{12}\theta_m^{[1]}(b)-R_{22}\theta_m(b)+R_{21}\phi_m(b)-R_{11}\phi_m^{[1]}(b)\big).
\end{align}
Hence, assertions \eqref{4.34} and \eqref{4.35} follow from Theorem \ref{t4.1} with
\begin{align}
    \begin{split}
        a_0&= e^{i\varphi}\big(R_{12}\theta_0^{[1]}(b)-R_{22}\theta_0(b)+R_{21}\phi_0(b)-R_{11}\phi_0^{[1]}(b)\big)+e^{2i\varphi}+1,\\
        a_k&=e^{i\varphi}\big(R_{12}\theta_k^{[1]}(b)-R_{22}\theta_k(b)+R_{21}\phi_k(b)-R_{11}\phi_k^{[1]}(b)\big),\quad k\in\N.
    \end{split}
\end{align}
\end{proof}

Next, we provide corollaries regarding the most common coupled boundary conditions, periodic and antiperiodic as well as the Krein-von Neumann extension.

\begin{corollary}[Periodic boundary conditions]
\lb{c4.8}
Assume Hypothesis \ref{h2.1}, consider $T_{0,I_2}$ as described in Theorem \ref{t2.2}\,$(ii)$, and let $m_0=0,1$, denote the multiplicity of zero as an eigenvalue of $T_{0,I_2}$. Then, 
\begin{align}\lb{4.40}
    \zeta(n;T_{0,I_2})=-\Res \left[ z^{-n}\dfrac{d}{dz}\ln (F_{0,I_2}(z));\ z=0 \right]=-n \, b_n, \quad n \in \bbN, 
\end{align}
where for $m_0=0$,
\begin{align}\lb{4.41}
    b_1&=\big[-\theta_1(b)-\phi^{[1]}_1(b)\big] \big/ \big[-\theta_0(b)-\phi^{[1]}_0(b)+2\big],    \\
    b_j&=\dfrac{-\theta_j(b)-\phi^{[1]}_j(b)}{-\theta_0(b)-\phi^{[1]}_0(b)+2}-\sum_{\ell=1}^{j-1} \frac{\ell}{j} \, 
    \dfrac{-\theta_{j-\ell}(b)-\phi^{[1]}_{j-\ell}(b)}{-\theta_0(b)-\phi^{[1]}_0(b)+2}b_{\ell}, 
    \quad j \in \bbN, \; j\geq 2,   \no 
\end{align}
and for $m_0=1$,
\begin{align}\lb{4.42}
    b_1&=  \big[\theta_{2}(b)+\phi^{[1]}_{2}(b)\big] \big/ \big[\theta_{1}(b)+\phi^{[1]}_{1}(b)\big],    \\
    b_j&= \dfrac{\theta_{j+1}(b)+\phi^{[1]}_{j+1}(b)}{\theta_{1}(b)+\phi^{[1]}_{1}(b)}-\sum_{\ell=1}^{j-1} 
    \frac{\ell}{j} \, 
    \dfrac{\theta_{j-\ell+1}(b)+\phi^{[1]}_{j-\ell+1}(b)}{\theta_{1}(b)+\phi^{[1]}_{1}(b)} b_{\ell},\quad j \in \bbN, \; j\geq 2.   \no
\end{align}
In particular, if zero is not an eigenvalue of $T_{0,I_2}$, then
\begin{align}
\text{\rm tr}_{L^2_r((a,b))}\big(T^{-1}_{0,I_2}\big)=\zeta(1;T_{0,I_2})
= \big[\theta_1(b)+\phi^{[1]}_1(b)\big] \big/ \big[-\theta_0(b)-\phi^{[1]}_0(b)+2\big].
\end{align}
\end{corollary}
\begin{proof}
Take $\varphi=0$ and $R=I_2$ in Theorem \ref{t4.7}.
\end{proof}

\begin{corollary}[Antiperiodic boundary conditions]
\lb{c4.9}
Assume Hypothesis \ref{h2.1}, consider $T_{0,-I_2}$ as described in Theorem \ref{t2.2}\,$(ii)$, and let $m_0=0,1$, denote the multiplicity of zero as an eigenvalue of $T_{0,-I_2}$. Then,
\begin{align}\lb{4.44}
    \zeta(n;T_{0,-I_2})=-\Res \left[ z^{-n}\dfrac{d}{dz}\ln (F_{0,-I_2}(z));\ z=0 \right]=-n \, b_n, \quad n \in \bbN, 
\end{align}
where for $m_0=0$,
\begin{align}\lb{4.45}
    b_1&= \big[\theta_1(b)+\phi^{[1]}_1(b)\big] \big/ \big[\theta_0(b)+\phi^{[1]}_0(b)+2\big],   \\
    b_j&= \dfrac{\theta_j(b)+\phi^{[1]}_j(b)}{\theta_0(b)+\phi^{[1]}_0(b)+2}-\sum_{\ell=1}^{j-1}\frac{\ell}{j} \, \dfrac{\theta_{j-\ell}(b)+\phi^{[1]}_{j-\ell}(b)}{\theta_0(b)+\phi^{[1]}_0(b)+2}b_{\ell}, 
    \quad j \in \bbN, \; j\geq 2,    \no 
\end{align}
and for $m_0=1$,
\begin{align}\lb{4.46}
    b_1&= \big[\theta_{2}(b)+\phi^{[1]}_{2}(b)\big] \big/ \big[\theta_{1}(b)+\phi^{[1]}_{1}(b)\big],\\
    b_j&= \dfrac{\theta_{j+1}(b)+\phi^{[1]}_{j+1}(b)}{\theta_{1}(b)+\phi^{[1]}_{1}(b)}-\sum_{\ell=1}^{j-1}
    \frac{\ell}{j} \,
    \dfrac{\theta_{j-\ell+1}(b)+\phi^{[1]}_{j-\ell+1}(b)}{\theta_{1}(b)+\phi^{[1]}_{1}(b)} b_{\ell},\quad j \in \bbN, \; j\geq 2.   \no 
\end{align}
In particular, if zero is not an eigenvalue of $T_{0,-I_2}$, then
\begin{align}
\text{\rm tr}_{L^2_r((a,b))}\big(T^{-1}_{0,-I_2}\big)=\zeta(1;T_{0,-I_2}) 
=- \big[\theta_1(b)+\phi^{[1]}_1(b)\big] \big/ \big[\theta_0(b)+\phi^{[1]}_0(b)+2\big].
\end{align}
\end{corollary}
\begin{proof}
Take $\varphi= 0$ and $R= - I_2$ in Theorem \ref{t4.7}.
\end{proof}

\begin{corollary}[Krein-von Neumann extension]
\lb{c4.10}~Assume Hypothesis \ref{h2.1}, consider $T_{0,R_K}$ the Krein-von Neumann extension of $T_{min}$ with 
\begin{align}\lb{4.48}
\varphi = 0, \quad R_K=\begin{pmatrix}
\theta(0,b,a) & \phi(0,b,a)\\
\theta^{[1]}(0,b,a) & \phi^{[1]}(0,b,a)
\end{pmatrix},
\end{align}
and let $m_0=2$, denote the multiplicity of zero as an eigenvalue of $T_{0,R_K}$. Then,
\begin{align}\lb{4.49}
    \zeta(n;T_{0,R_K})=-\Res \left[ z^{-n}\dfrac{d}{dz}\ln (F_{0,R_K}(z));\ z=0 \right]=-n \, b_n, \quad n \in \bbN, 
\end{align}
where
\begin{align}
        &\no  b_1=\dfrac{\phi_0(b)\theta_{3}^{[1]}(b)-\phi_0^{[1]}(b)\theta_{3}(b)+\theta_0^{[1]}(b)\phi_{3}(b)-\theta_0(b)\phi_{3}^{[1]}(b)}{\phi_0(b)\theta_{2} ^{[1]}(b)-\phi_0^{[1]}(b)\theta_{2} (b)+\theta_0^{[1]}(b)\phi_{2}(b)-\theta_0(b)\phi_{2}^{[1]}(b)},\\
        \no \\
        &\no  b_j=\dfrac{\phi_0(b)\theta_{j+2}^{[1]}(b)-\phi_0^{[1]}(b)\theta_{j+2}(b)+\theta_0^{[1]}(b)\phi_{j+2}(b)-\theta_0(b)\phi_{j+2}^{[1]}(b)}{\phi_0(b)\theta_{2} ^{[1]}(b)-\phi_0^{[1]}(b)\theta_{2} (b)+\theta_0^{[1]}(b)\phi_{2}(b)-\theta_0(b)\phi_{2}^{[1]}(b)}\\
        &\no \resizebox{\textwidth}{!}{$\quad\quad\ \ -\displaystyle\sum_{\ell=1}^{j-1}\frac{\ell}{j} \, 
        \dfrac{\phi_0(b)\theta_{j-\ell+2}^{[1]}(b)-\phi_0^{[1]}(b)\theta_{j-\ell+2}(b)+\theta_0^{[1]}(b)\phi_{j-\ell+2}(b)-\theta_0(b)\phi_{j-\ell+2}^{[1]}(b)}{\phi_0(b)\theta_{2} ^{[1]}(b)-\phi_0^{[1]}(b)\theta_{2} (b)+\theta_0^{[1]}(b)\phi_{2}(b)-\theta_0(b)\phi_{2}^{[1]}(b)} b_{\ell},$}\\
        &\hspace{9cm} j \in \bbN, \; j\geq 2.\lb{4.50}
\end{align}
\end{corollary}
\begin{proof}
As shown in \cite[Example 3.3]{CGNZ14}, the resulting operator $T_{0,R_K}$ represents the Krein--von Neumann extension of $T_{min}$. Take $\varphi=0$ and $R=R_K$ (as defined by \eqref{4.48}) in Theorem \ref{t4.7}, denoting 
$\phi_0(b) = \phi(0,b,a)$, $\phi_0^{[1]}(b) = \phi^{[1]}(0,b,a)$, $\theta_0(b)= \theta(0,b,a)$, and 
$\theta_0^{[1]}(b) = \theta^{[1]}(0,b,a)$ as before, for simplicity.
\end{proof}

\section{Examples} \lb{s5}

In this section, we provide an array of examples illustrating our approach for computing spectral $\zeta$-function values of regular Schr\"odinger operators starting with the simplest case of $q=0$, then a positive (piecewise) constant potential, followed by a constant negative potential, and ending with the case of a linear potential. 

Throughout this section we suppose that 
\begin{equation}
p=r=1 \, \text{ a.e.~on } \, (a,b) 
\end{equation}
which leaves the potential coefficient $q \in L^1((a,b);dx)$, $q$ real-valued, and hence leaves us with the differential expression 
\begin{equation}
\tau = - \big(d^2/dx^2\big) + q(x), \quad x \in (a,b).
\end{equation}

\subsection{The Example \textit{q}=0} \lb{s5.1}  
\hfill

We start by providing examples for calculating spectral $\zeta$-function values for the simple case $q(x)=0$, $x\in(a,b)$, imposing various boundary conditions. In this case $\tau y=-y''=z y$ has the following linearly independent solutions,
\begin{align}
\lb{5.2}
    \phi(z,x,a)=z^{-1/2} \sin\big(z^{1/2} (x-a)\big), 
    \quad \theta(z,x,a)=\cos\big(z^{1/2} (x-a)\big),\quad z\in\C .
\end{align}
Hence, 
\begin{align}
\phi(z,b,a)& =\sum_{m=0}^\infty z^m\phi_m(b),\quad z\in\C, \quad 
\phi_k(b)=\dfrac{(-1)^{k}}{(2k+1)!}(b-a)^{2k+1}, \; k\in\N,  \no \\
\theta(z,b,a) &=\sum_{m=0}^\infty z^m\theta_m(b),\quad z\in\C, \quad 
\theta_k(b) = \dfrac{(-1)^{k}}{(2k)!}(b-a)^{2k}, \; k\in\N, \\
\phi'(z,b,a) &= \sum_{m=0}^\infty z^m\phi'_m(b),\quad z\in\C, \quad 
\phi'_k(b) =-\dfrac{(-1)^k}{(k+1)!}(b-a)^{k+1},\quad \ k\in\N,   \no \\
\theta'(z,b,a) &=\sum_{m=0}^\infty z^m\theta'_m(b),\quad z\in\C, \quad 
\theta'_k(b) =\dfrac{(-1)^{k}}{k!}(b-a)^k,\quad \ k\in\N. 
\end{align}

One can explicitly write the corresponding expressions for $F_{\a,\b}(z)$ and $F_{\varphi, R}(z)$ for this example to find for $\a,\b\in[0,\pi)$,
\begin{align}
&F_{\a,\b}(z) = \cos(\a)\big[-\sin(\b)\ \cos\big(z^{1/2} (b-a)\big)+\cos(\b) z^{-1/2} \sin\big(z^{1/2} (b-a)\big)\big]   \no \\
& \quad -\sin(\a)\big[\sin(\b)\ z^{1/2} \sin\big(z^{1/2} (b-a)\big)+\cos(\b)\ \cos\big(z^{1/2} (b-a)\big)\big],
\end{align}
and for $\varphi\in[0,\pi),\ R\in SL(2,\R)$,
\begin{align}
& F_{\varphi,R}(z) = e^{i\varphi}\big[-R_{12}z^{1/2} \sin\big(z^{1/2} (b-a)\big)-R_{22}\cos\big(z^{1/2} (b-a)\big)   \no \\
& \quad +R_{21} z^{-1/2} \sin\big(z^{1/2} (b-a)\big) -R_{11}\cos\big(z^{1/2} (b-a)\big)\big]+e^{2i\varphi}+1.
\end{align}

We provide an explicit expression for $\zeta(1;T_{A,B})$
since it only involves the first few coefficients of the small-$z$ expansion. In the case of separated boundary conditions 
one obtains
\begin{align}
\begin{split}
a_{0}&= \cos (\a ) ((b-a) \cos (\b )-\sin (\b ))-\sin (\a ) \cos (\b ),\\
a_{1}&= \cos (\a ) \left(\frac{1}{2} (b-a)^2 \sin (\b )-\frac{1}{6} (b-a)^3 \cos (\b )\right)\\
&\quad +\sin (\a ) \left(\frac{1}{2} (b-a)^2 \cos (\b )-(b-a) \sin (\b )\right),\\
a_{2}&= \sin (\a ) \left(\frac{1}{6} (b-a)^3 \sin (\b )-\frac{1}{24} (b-a)^4 \cos (\b )\right)\\
&\quad +\cos (\a ) \left(\frac{1}{120} (b-a)^5 \cos (\b )-\frac{1}{24} (b-a)^4 \sin (\b )\right).
\end{split}
\end{align}
If $T_{\a,\b}$ does not have a zero eigenvalue, then $a_{0}\neq 0$ and, hence, one finds from \eqref{4.4},
\begin{align}
&\text{\rm tr}_{L^2_r((0,b))}\big(T_{\a,\b}^{-1}\big)=\zeta(1;T_{\a,\b})=    \\
&\resizebox{\textwidth}{!}{$\dfrac{\cos (\a ) \left(3 (b-a)^2 \sin (\b )-(b-a)^3 \cos (\b )\right)+\sin (\a ) \left(3(b-a)^2 \cos (\b )-6(b-a) \sin (\b )\right)}{6\sin (\a ) \cos (\b )-6\cos (\a ) ((b-a) \cos (\b )-\sin (\b ))}.$} \no 
\end{align}
If, instead, $T_{\a,\b}$ has a zero eigenvalue then $a_{0}=0$ and one finds
\begin{align}
&\zeta(1;T_{\a,\b})=   \\
&\resizebox{\textwidth}{!}{$\dfrac{-\sin (\a ) \left(20 (b-a)^3 \sin (\b )-5 (b-a)^4 \cos (\b )\right)-\cos (\a ) \left( (b-a)^5 \cos (\b )-5 (b-a)^4 \sin (\b )\right)}{\cos (\a ) \left(60 (b-a)^2 \sin (\b )-20 (b-a)^3 \cos (\b )\right)+\sin (\a ) \left(60 (b-a)^2 \cos (\b )-120(b-a) \sin (\b )\right).}$}    \no 
\end{align} 

In the case of coupled boundary conditions one finds
\begin{align}
a_{0}&= e^{i \varphi } ((b-a)R_{21}-R_{11}-R_{22})+e^{2 i \varphi }+1,    \no \\
a_{1}&= e^{i \varphi } \left(-\frac{1}{6} (b-a)^3R_{21}+\frac{1}{2} (b-a)^2R_{11}+\frac{1}{2} (b-a)^2R_{22}+ (a-b)R_{12}\right),    \\
a_{2}&= e^{i \varphi } \left(\frac{1}{120} (b-a)^5R_{21}-\frac{1}{24}(b-a)^4R_{11}-\frac{1}{24} (b-a)^4R_{22}+\frac{1}{6} (b-a)^3R_{12}\right).   \no 
\end{align}
Once again, if zero is not an eigenvalue of $T_{\varphi,R}$, $a_{0}\neq0$ and one finds
\begin{align}
\begin{split} 
&\text{\rm tr}_{L^2_r((0,b))}\big(T_{\varphi,R}^{-1}\big) = \zeta(1;T_{\varphi,R})   \\ 
& \quad 
=\dfrac{e^{i \varphi } \left(R_{21} (b-a)^3-3(b-a)^2R_{11}-3 (b-a)^2R_{22}+6 (b-a)R_{12}\right)}{6e^{i \varphi } ((b-a)R_{21}-R_{11}-R_{22})+6e^{2 i \varphi }+6}.
\end{split} 
\end{align}
If, on the other hand, zero is an eigenvalue of $T_{\varphi,R}$ with multiplicity one, then $a_{0}=0$ and 
\begin{align}
\begin{split}
\zeta(1;T_{\varphi,R})=\dfrac{(b-a)^5R_{21}-5(b-a)^4R_{11}-5 (b-a)^4R_{22}+20 (b-a)^3R_{12}}{20 (b-a)^3R_{21}-60 (b-a)^2R_{11}-60 (b-a)^2R_{22}+ 120(b-a)R_{12}}.
\end{split}
\end{align}
If zero is an eigenvalue of $T_{\varphi,R}$ with multiplicity two, we refer to the Krein--von Neumann extension, see Example \ref{e5.5}.

Finally we give the form of the zeta regularized functional determinant for this example.
As $z\downarrow 0$, one obtains
\begin{align}\lb{5.13}
F_{\a,\b}(z)=(b-a)\cos(\a)\cos(\b)-\sin(\a+\b)+\Oh(z),
\end{align}
which implies that for particular values of $\a$ and $\b$ one finds a zero eigenvalue. 
For now we will assume that no zero eigenvalue is present and hence we consider the following set of parameters
\begin{align}
\cA=\{(\a,\b) \in (0,\pi) \times (0,\pi) \, | \, (b-a)\cos(\a)\cos(\b)-\sin(\a+\b)\neq 0\}.
\end{align} 
For $(\a,\b) \in \cA$ one infers, by construction, that $m_0=0$ and hence, $\sin(\a)\sin(\b)\neq 0$. The latter condition implies that in \eqref{3.85} one must set $k_0=-2$. By using \eqref{5.13}, one obtains
\begin{align}
\begin{split}
\zeta'(0;T_{\a,\b})&=-\ln\left(\left|\frac{2 \cF_{\a,\b}(0)}{\sin(\a)\sin(\b)}\right|\right)\\
&=-\ln\left(\left|\frac{2(b-a)\cos(\a)\cos(\b)-2\sin(\a+\b)}{\sin(\a)\sin(\b)}\right|\right),
\end{split}
\end{align}
which coincides with \cite[Eq. (3.72)]{GK19}.

Furthermore, as $z\downarrow 0$, one obtains
\begin{align}\lb{5.17}
F_{\varphi,R}(z)=e^{i\varphi}[(b-a)R_{21}-R_{11}-R_{22}]+e^{2i\varphi}+1+\Oh(z),
\end{align}
which implies that for particular choices of $\varphi$ and $R$ one finds a zero eigenvalue. 
For now we will assume that no zero eigenvalue is present and hence we consider the following set of parameters
\begin{align}
\cB=\{\varphi\in(0,\pi),R\in SL(2,\bbR) \, | \, e^{i\varphi}[(b-a)R_{21}-R_{11}-R_{22}]+e^{2i\varphi}+1\neq 0\}.
\end{align} 
For $(\varphi, R) \in \cB$ we have, by construction, that $m_0=0$. Making the additional assumption $R_{12}\neq0$ implies that in \eqref{3.85} one must set $k_0=-2$. By using \eqref{5.17}, one obtains
\begin{align}
\begin{split}
\zeta'(0;T_{\varphi,\wti{R}})&=-\ln\left(\left|2\cF_{\varphi,\wti{R}}(0)/R_{12}\right|\right)\\
&=-\ln\left(\left|\frac{2[(b-a)R_{21}-R_{11}-R_{22}]+4\cos(\varphi)}{R_{12}}\right|\right).
\end{split}
\end{align}
If $R_{12}=0$, then since $R\in SL(2,\bbR)$, by assumption $R_{11}\neq-R_{22}$ which implies that in \eqref{3.85} one must set $k_0=-1$. By once again using \eqref{5.17}, one obtains
\begin{align}
\begin{split}
\zeta'(0;T_{\varphi,\wti{R}})&=-\ln\left(\left|\frac{2\cF_{\varphi,\wti{R}}(0)}{R_{11}+R_{22}}\right|\right)\\
&=-\ln\left(\left|\frac{2[(b-a)R_{21}-R_{11}-R_{22}]+4\cos(\varphi)}{R_{11}+R_{22}}\right|\right).
\end{split}
\end{align}

The following examples, each with different boundary conditions, will illustrate how the main theorems and corollaries of the previous section can be used to effectively compute the spectral $\zeta$-function values of the operator, $T_{A,B}$, for $n\in\N$.

\begin{example}[Dirichlet boundary conditions]\lb{e5.1}
Consider the case $\a=\b=0$. Then the operator $T_{0,0}$ has eigenvalues and eigenfunctions given by
\begin{align}
    \lambda_k= k^2\pi^2\big/(b-a)^2,\quad y_k(x)= \lambda_k^{-1/2}\sin\big(\lambda_k^{1/2}(x-a)\big),\quad k\in\N
\end{align}
$($in particular, $z=0$ is not an eigenvalue of $T_{0,0}$$)$, and
\begin{align}
F_{0,0}(z)= z^{-1/2} \sin\big(z^{1/2} (b-a)\big), \quad z\in\C.
\end{align}
Applying Corollary \ref{c4.3} with $m_0=0$ one finds for $n=1,2,3,4$,
\begin{align}
\begin{split}
&\zeta(1;T_{0,0})= (b-a)^2 \pi^{-2} \sum_{k=1}^\infty k^{-2} =\text{\rm tr}_{L^2_r((a,b))}\big(T^{-1}_{0,0}\big)=(b-a)^2/6,\\
&\zeta(2;T_{0,0})=(b-a)^4/90,\\
&\zeta(3;T_{0,0})=(b-a)^6/945,\\
&\zeta(4;T_{0,0})=(b-a)^8/9450.    \lb{5.21} 
\end{split}
\end{align}

Next, we explicitly compute the zeta regularized functional determinant with Dirichlet boundary conditions. Since no zero eigenvalue is present and $\Gamma_{0}=-(b-a)$, one obtains
\begin{align}
\zeta'(0;T_{0,0})=-\ln [2F_{0,0}(0)]=-\ln[2(b-a)].
\end{align}
\end{example}

One can corroborate the values found in Example \ref{e5.1} by utilizing the following relation of $\zeta(s;T_{0,0})$ with the Riemann $\zeta$-function (see, e.g., \cite{Ay74}, \cite{Di55} for some background)
\begin{align}
\zeta(s;T_{0,0})= (b-a)^{2s} \pi^{-2s} \zeta(2s),\quad \Re(s)> 1/2.
\end{align}
By using \cite[0.2333]{GR80}, the last expression allows us to find for $s=n\in\N$,
\begin{align}
\zeta(n;T_{0,0}) = 2^{2n-1}(b-a)^{2n} |B_{2n}|/[(2n)!], 
\end{align}
where $B_{2n}$ is the $2n$th Bernoulli number (cf. \cite[Ch. 23]{AS72}).

\begin{example}[Neumann boundary conditions]\lb{e5.2}
Consider the case $\a=\b=\pi/2$. Then the operator $T_{\pi/2,\pi/2}$ has eigenvalues and eigenfunctions given by
\begin{align}
    \lambda_k=k^2\pi^2/(b-a)^2,\quad y_k(x)=\cos\big(\lambda_k^{1/2}(x-a)\big),\quad k\in\N_0
\end{align}
$($in particular, $z=0$ is a simple eigenvalue of $T_{\pi/2,\pi/2}$$)$ and
\begin{align}
F_{\pi/2,\pi/2}(z)=-z^{1/2} \sin\big(z^{1/2} (b-a)\big),\quad z\in\C.
\end{align} 
Applying Corollary \ref{c4.6} with $m_0=1$ one finds for $n=1,2,3,4$,
\begin{align}
\begin{split}
&\zeta(1;T_{\pi/2,\pi/2})= (b-a)^2 \pi^{-2} \sum_{k=1}^\infty k^{-2}=(b-a)^2/6,\\
&\zeta(2;T_{\pi/2,\pi/2})= (b-a)^4/90,\\
&\zeta(3;T_{\pi/2,\pi/2})= (b-a)^6/945,\\
&\zeta(4;T_{\pi/2,\pi/2})= (b-a)^8/9450.    \lb{5.29}
\end{split}
\end{align}
\end{example}

Noting that the series expression for $\zeta(s;T_{\pi/2,\pi/2})$ in \eqref{2.39} sums only over non-zero eigenvalues, and that the eigenvalues for Dirichlet and Neumann boundary conditions only differ by zero being an eigenvalue for the latter, but not the former, the same expressions apply as in Example \ref{e5.1}, which is reflected in equations \eqref{5.21} and \eqref{5.29} yielding the same values. 

\begin{example}[Periodic boundary conditions]\lb{e5.3}
Consider the case $\varphi=0$, $R=I_2$. Then the operator $T_{0,I_2}$ has eigenvalues given by
\begin{align}
\lambda_k=(2k)^2\pi^2/(b-a)^2,\quad k\in\N_0.
\end{align}In particular, $z=0$ is a simple eigenvalue of $T_{0,I_2}$ and all other eigenvalues of $T_{0,I_2}$ are of multiplicity 2, and
\begin{align}
F_{0,I_2}(z)=-2\cos\big(z^{1/2} (b-a)\big)+2,\quad z\in\C.
\end{align}
Applying Corollary \ref{c4.8} with $m_0=1$ one finds for $n=1,2,3,4$,
\begin{align}
\begin{split}
&\zeta(1;T_{0,I_2})=2 (b-a)^2 \pi^{-2} \sum_{k=1}^\infty (2k)^{-2} = (b-a)^2/12,\\
&\zeta(2;T_{0,I_2})= (b-a)^4/720,\\
&\zeta(3;T_{0,I_2})= (b-a)^6/30240,\\
&\zeta(4;T_{0,I_2})=(b-a)^8/1209600.
\end{split}
\end{align}
\end{example}

Here, once again, one can verify the values found in Example \ref{e5.3} by utilizing the following relation of $\zeta(s;T_{0,I_2})$ with the Riemann $\zeta$-function,
\begin{align}
\zeta(s;T_{0,I_2})=2^{1-2s}\pi^{-2s} (b-a)^{2s}\zeta(2s),\quad \Re(s)>1/2.
\end{align}
By using \cite[0.2333]{GR80}, the last expression allows one to find for $s=n\in\N$,
\begin{align}
\zeta(n;T_{0,I_2})=(b-a)^{2n}|B_{2n}|/[(2n)!].
\end{align}

\begin{example}[Antiperiodic boundary conditions]\lb{e5.4}
Consider the case $\varphi= 0$, $R= - I_2$. Then the operator $T_{0,-I_2}$ has eigenvalues given by
\begin{align}
\lambda_k=(2k-1)^2\pi^2/(b-a)^2,\quad k\in\N.
\end{align}
In particular, $z=0$ is not an eigenvalue of $T_{0,-I_2}$ and all eigenvalues of $T_{0,-I_2}$ are of multiplicity 2, and
\begin{align}
F_{0,-I_2}(z)=2\cos\big(z^{1/2} (b-a)\big)+2,\quad z\in\C.
\end{align}
Applying Corollary \ref{c4.9} with $m_0=0$ one finds for $n=1,2,3,4$,
\begin{align}
\begin{split}
&\zeta(1;T_{0,-I_2})=2 (b-a)^2 \pi^{-2} \sum_{k=1}^\infty (2k-1)^{-2}
=\text{\rm tr}_{L^2_r((a,b))}\big(T^{-1}_{0,-I_2}\big)=(b-a)^2/4,\\
&\zeta(2;T_{0,-I_2})= (b-a)^4/48,\\
&\zeta(3;T_{0,-I_2})=(b-a)^6/480,\\
&\zeta(4;T_{0,-I_2})= [17/80640] (b-a)^8.
\end{split}
\end{align}
\end{example}

One can verify the values found in Example \ref{e5.4} by utilizing the following relation,
\begin{align}
& \zeta(s;T_{0,-I_2})=2(b-a)^{2s}\pi^{-2s} \sum_{k\in\N}(2k-1)^{-2s}=\big(1-2^{-2s}\big) 2(b-a)^{2s}\pi^{-2s} \zeta(2s),    \no \\ 
& \hspace*{9cm} \Re(s)>1/2,
\end{align}
which in turn by using either \cite[0.2335]{GR80} on the first equality or \cite[0.2333]{GR80} on the second allows one to find for $s=n\in\N$,
\begin{align}
\zeta(n;T_{0,-I_2})= (2^{2n}-1)(b-a)^{2n} |B_{2n}|/[(2n)!].
\end{align}

\begin{example}[Krein--von Neumann boundary conditions]\lb{e5.5}
Consider the case $\varphi=0$, $R=R_K$, with
\begin{align}\lb{5.40}
R_K=\begin{pmatrix}
\theta(0,b,a) & \phi(0,b,a)\\
\theta^{[1]}(0,b,a) & \phi^{[1]}(0,b,a)
\end{pmatrix}=\begin{pmatrix}
1 & b-a\\0 & 1
\end{pmatrix}.
\end{align}
As shown in \cite[Example 3.3]{CGNZ14}, the resulting operator $T_{0,R_K}$ represents the Krein--von Neumann extension of $T_{min}$. For more on the Krein--von Neumann extension, including an extensive discussion of eigenvalues and eigenfunctions, see \cite{AS80} or \cite{AGMT10}. From \eqref{2.17} with $\varphi=0$, $R=R_K$ defined as in \eqref{5.40},
\begin{align}\lb{5.41}
F_{0,R_K}(z)=(a-b)z^{1/2} \sin\big(z^{1/2} (b-a)\big)-2\cos\big(z^{1/2} (b-a)\big)+2,\quad z\in\C.
\end{align}
Using the series expansions in \eqref{5.41}, one finds
\begin{align}
F_{0,R_K}(z)\underset{z\downarrow0}{=} \big[(b-a)^4/12\big] z^2+\Oh\big(z^3\big),
\end{align}
so that $z=0$ is a zero of multiplicity two of $F_{0,R_K}(z)$ and hence an eigenvalue of multiplicity two of $T_{0,R_K}$ $($coinciding with what was found in \cite{AGMT10} and noted in \cite[Example 3.7]{GK19}$)$. Applying Corollary \ref{c4.10} with $m_0=2$ gives
\begin{align}
\begin{split}
\zeta(1;T_{0,R_K})&= (b-a)^2/15,\\
\zeta(2;T_{0,R_K})&= [11/12600] (b-a)^4,\\
\zeta(3;T_{0,R_K})&= (b-a)^6/54000,\\
\zeta(4;T_{0,R_K})&=[457/317520000] (b-a)^8.
\end{split}
\end{align}
\end{example}

\subsection{Examples of Nonnegative (Piecewise) Constant Potentials} \lb{s5.2}
\hfill

Next we provide examples for calculating spectral $\zeta$-function values considering a positive (piecewise) constant  potential $q$, imposing Dirichlet boundary conditions. 

\begin{example}\lb{e5.6}
Let $V_0 \in (0,\infty)$, consider $q(x) = V_0$, $x\in(a,b)$, and denote by $T_{0,0}$ the associated 
Schr\"odinger operator with Dirichlet boundary conditions at $a$ and $b$ 
$($i.e., $\a=\b=0$$)$. Then,
\begin{align}
\begin{split}\lb{5.46}
& \phi(z,x,a)= (z-V_0)^{-1/2} \sin\big((z-V_0)^{1/2}(x-a)\big),     \\ 
& \theta(z,x,a)=\cos\big((z-V_0)^{1/2}(x-a)\big),\quad x \in (a,b), \;  z\in\C .
\end{split}
\end{align}
Furthermore, the eigenvalues and eigenfunctions for $T_{0,0}$ with $q(x)=V_0>0,\ x\in(a,b),$ are given by
\begin{align}
\begin{split} 
& \lambda_k= k^2\pi^2/(b-a)^{-2}+V_0,     \\
& y_k(x) = (\lambda_k-V_0)^{-1/2} \sin\big((\lambda_k-V_0)^{1/2}(x-a)\big),\quad k\in\N
\end{split} 
\end{align}
$($in particular, $z=0$ is not an eigenvalue of $T_{0,0}$$)$, and
\begin{align}
F_{0,0}(z)= (z-V_0)^{-1/2} \sin\big((z-V_0)^{1/2}(b-a)\big),\quad z\in\C.
\end{align}
Applying Corollary \ref{c4.3} with $m_0=0$ one finds for $n=1,2,3$ $($the expression for $n=4$ is significantly longer and hence is omitted here$)$,
\begin{align}
\no  &\zeta(1;T_{0,0})=\sum_{k=1}^\infty \left[\frac{k^2\pi^2}{(b-a)^2}+V_0\right]^{-1}=\text{\rm tr}_{L^2_r((a,b))}\big(T^{-1}_{0,0}\big)\\
\no  &\hspace{1.38cm}= \big[V_0^{1/2}(b-a)\coth\big(V_0^{1/2}(b-a)\big)-1\big]\big/(2V_0),  \\
\no  &\resizebox{\textwidth}{!}{$\zeta(2;T_{0,0})=\dfrac{V_0^{1/2}(b-a)\sinh\big(2V_0^{1/2}(b-a)\big)-2\cosh\big(2V_0^{1/2}(b-a)\big)+2V_0(b-a)^2+2}{8V_0^2\sinh^2\big(V_0^{1/2}(b-a)\big)},$}\\
\no  &\zeta(3;T_{0,0})=\big(64V_0^3\sinh^2\big(V_0^{1/2}(b-a)\big)\big)^{-1}\big[12V_0(b-a)^2- 16\cosh\big(2V_0^{1/2}(b-a)\big)\\
\no  &\hspace{1.38cm}\quad+16+V_0^{1/2}(b-a)\big(8a^2V_0-16abV_0+8b^2V_0-3\big)\coth\big(V_0^{1/2}(b-a)\big)\\
\no  &\hspace{1.38cm}\quad-3aV_0^{1/2}\cosh\big(3V_0^{1/2}(b-a)\big)\big(\sinh\big(V_0^{1/2}(b-a)\big)\big)^{-1}\\
&\hspace{1.38cm}\quad+3bV_0^{1/2}\cosh\big(3V_0^{1/2}(b-a)\big)\big(\sinh\big(V_0^{1/2}(b-a)\big)\big)^{-1}\big].\lb{5.49}
\end{align}
\end{example}

Taking the limit $V_0\downarrow0$ of \eqref{5.49} recovers the expressions in Example \ref{e5.1}.

\begin{remark}\lb{r5.7}
One can also verify the expressions found in Example \ref{e5.6} by means of the one-dimensional Epstein $\zeta$-function given by
\begin{align}
\zeta_E(s;m^2)=\sum_{k=-\infty}^\infty \big(k^2+m^2\big)^{-s}, \quad m^2\neq0, \; s > 1/2
\end{align}
(see, e.g., the classical sources \cite{Ep03}, \cite{Ep07}, \cite{Ka06}, and \cite[Sect. 1.1.3]{El12}, 
\cite[Sects.~1.2.2, 5.3.2]{EORBZ94}, \cite{Ki94}, \cite[Ch.~3, App.~A]{Ki02},  and the extensive list of references therein).
Now one finds that $\zeta(s;T_{0,0})$ in Example \ref{e5.6} can be written in the form 
\begin{align}\lb{5.51}
\begin{split}
\zeta(s;T_{0,0})&= \sum_{k=1}^\infty \left[\frac{k^2\pi^2}{(b-a)^2}+V_0\right]^{-s} 
= (b-a)^{2s}\pi^{-2s} \sum_{k=1}^\infty \left[k^2+m^2\right]^{-s}\\
&= 2^{-1} (b-a)^{2s} \pi^{-2s} \big[\zeta_E(s;m^2)-m^{-2s}\big], \quad s > 1/2, 
\end{split}
\end{align}
where
\begin{align}
m^2=(b-a)^2V_0 \pi^{-2}>0.    \lb{5.52} 
\end{align}
Then the following formula for the analytic continuation of $\zeta_E(s;m^2)$ in $s$ 
for $m\neq0,-1,-2,\ldots$ (see \cite[Sect. 4.1.1]{El12})
\begin{align}\lb{5.53}
\begin{split}
\zeta_E(s;m^2)=\pi^{1/2}\dfrac{\Gamma(s-\frac{1}{2})}{\Gamma(s)}m^{1-2s}+\dfrac{4\pi^s}{\Gamma(s)}m^{1/2-s}\sum_{n=1}^\infty n^{s-1/2}K_{s-1/2}(2\pi mn),\\
s\neq (1/2)-\ell,\ \ell\in\bbN_0,\ s \in\bbC,
\end{split} 
\end{align}
where $K_{\mu}(\dott)$ is the modified Bessel function of the second kind (see for example \cite[Chs. 9-10]{AS72}), can be used to explicitly verify the expressions found in Example \ref{e5.6}.

We verify the expressions for $\zeta(1;T_{0,0})$ and $\zeta(2;T_{0,0})$ next. From \eqref{5.53} one has, using the fact that $K_{1/2}(z)=\pi^{1/2}(2z)^{-1/2}e^{-z}$,
\begin{align}
\no  \zeta_E(1;m^2)&=\pi m^{-1}+4\pi m^{-1/2}\sum_{n=1}^\infty n^{1/2} \pi^{1/2}(4\pi mn)^{-1/2} e^{-2\pi mn}\\
\no  &=\pi m^{-1}+2\pi m^{-1}\sum_{n=1}^\infty e^{-2\pi mn} =\pi m^{-1}+2\pi m^{-1}\dfrac{1}{e^{2\pi m}-1}\\
&=\dfrac{\pi}{m}\coth(\pi m).   \lb{5.52a}
\end{align}
Thus, from \eqref{5.51} and \eqref{5.52} one obtains
\begin{align}
\no  \zeta(1;T_{0,0})&=\dfrac{(b-a)^{2}}{2\pi^{2}}\big(\zeta_E(1;m^2)-m^{-2}\big)=\dfrac{(b-a)^{2}}{2\pi^{2}}\bigg(\dfrac{\pi m\coth(\pi m)-1}{m^2}\bigg)\\
&=\big[V_0^{1/2}(b-a)\coth\big(V_0^{1/2}(b-a)\big)-1\big]\big/(2V_0),\lb{5.55}
\end{align}
in accordance with Example \ref{e5.6}.

Next we verify the expression for $\zeta(2;T_{0,0})$ by first noting that
\begin{align}
\dfrac{d}{dm}\big(\zeta_E(s;m^2)\big)=-2sm \zeta_E(s+1;m^2),
\end{align}
which implies the functional equation
\begin{align}
\zeta_E(s+1;m^2)=-\dfrac{1}{2sm}\dfrac{d}{dm}\big(\zeta_E(s;m^2)\big).  \lb{5.55a}
\end{align}
From \eqref{5.52a} and \eqref{5.55a} with $s=1$ one has
\begin{equation}
 \zeta_E(2;m^2) = - \dfrac{\pi}{2m}\dfrac{d}{dm}\bigg(\dfrac{\coth(\pi m)}{m}\bigg) 
=\dfrac{\pi\sinh(2\pi m)+2\pi^2 m}{4m^3\sinh^2(\pi m)}.
\end{equation}
Thus from \eqref{5.51} and \eqref{5.52} one obtains
\begin{align}
\no  &\zeta(2;T_{0,0})=\dfrac{(b-a)^{4}}{2\pi^{4}}\big(\zeta_E(2;m^2)-m^{-4}\big)\\
\no  &\hspace{1.38cm}=\dfrac{(b-a)^{4}}{2\pi^{4}}\bigg(\dfrac{\pi\sinh(2\pi m)+2\pi^2 m}{4m^3\sinh^2(\pi m)}-\dfrac{1}{m^4}\bigg)\\
&\resizebox{\textwidth}{!}{$\hspace{1.53cm}=\dfrac{V_0^{1/2}(b-a)\sinh\big(2V_0^{1/2}(b-a)\big)-2\cosh\big(2V_0^{1/2}(b-a)\big)+2V_0(b-a)^2+2}{8V_0^2\sinh^2\big(V_0^{1/2}(b-a)\big)},$}
\lb{5.58}
\end{align}
again in accordance with Example \ref{e5.6}. All other positive integer values can be found recursively by means of \eqref{5.52a} and the functional equation \eqref{5.55a}.\hfill$\diamond$
\end{remark}

Next, we turn to the case of a nonnegative piecewise constant potential (a potential well): 

\begin{example} \lb{e5.8} 
Let $c, d \in (a,b)$, $c < d$, $V_0 \in (0,\infty)$, consider 
\begin{align}\lb{5.60}
q(x)=\begin{cases}
0 & x\in (a,c),\\
V_0 & x\in (c,d),\\
0 & x\in (d,b),
\end{cases}  
\end{align}
and denote by $T_{0,0}$ the associated Schr\"odinger operator with Dirichlet boundary conditions at $a$ and $b$. Then, 
for $z \in \bbC$, 
\begin{align}
    \phi(z,x,a)&= z^{-1/2} \sin\big(z^{1/2} (x-a)\big),  \quad x \in (a,c),   \no \\ 
    \theta(z,x,a)&=\cos\big(z^{1/2} (x-a)\big),\quad x\in (a,c),   \no \\
    \phi(z,x,a)&= \cos\big(z^{1/2} (c-a)\big) (z-V_0)^{-1/2} \sin\big((z-V_0)^{1/2}(x-c)\big)   \no \\
    &\quad+z^{-1/2}\sin\big(z^{1/2} (c-a)\big)\cos\big((z-V_0)^{1/2}(x-c)\big), \quad x \in (c,d),   \no \\
    \theta(z,x,a)&= -z^{1/2} \sin\big(z^{1/2} (c-a)\big) (z-V_0)^{-1/2} \sin\big((z-V_0)^{1/2}(x-c)\big)   \no \\
    &\quad+\cos\big(z^{1/2} (c-a)\big)\cos\big((z-V_0)^{1/2}(x-c)\big),\quad x\in (c,d),  \no \\
 \phi(z,x,a)&= \bigg[\cos\big(z^{1/2} (c-a)\big)\cos\big((z-V_0)^{1/2}(d-c)\big)   \no \\
& \quad \;\; -(z-V_0)^{1/2} z^{-1/2} \sin\big(z^{1/2} (c-a)\big)\sin\big((z-V_0)^{1/2}(d-c)\big)\bigg]   \no \\
& \quad \times z^{-1/2} \sin\big(z^{1/2} (x-d)\big)    \lb{5.61} \\
&+\bigg[\cos\big(z^{1/2} (c-a)\big) (z-V_0)^{-1/2} \sin\big((z-V_0)^{1/2}(d-c)\big)   \no \\
&\quad \;\; + z^{-1/2} \sin\big(z^{1/2} (c-a)\big) \cos\big((z-V_0)^{1/2}(d-c)\big)\bigg]\cos\big(z^{1/2} (x-d)\big), \no \\
& \hspace*{8cm} x \in (d,b),   \no \\
\theta(z,x,a)&= -\bigg[z^{1/2} \sin\big(z^{1/2} (c-a)\big)\cos\big((z-V_0)^{1/2}(d-c)\big)   \no \\
&\qquad \; +(z-V_0)^{1/2}\cos\big(z^{1/2} (c-a)\big)\sin\big((z-V_0)^{1/2}(d-c)\big)\bigg]      \no \\
& \qquad \; \times z^{-1/2} \sin\big(z^{1/2} (x-d)\big)    \no \\
&\quad+\bigg[-z^{1/2} \sin\big(z^{1/2} (c-a)\big) (z-V_0)^{-1/2} \sin\big((z-V_0)^{1/2}(d-c)\big)   \no \\
&\qquad \; +\cos\big(z^{1/2} (c-a)\big)\cos\big((z-V_0)^{1/2}(d-c)\big)\bigg]\cos\big(z^{1/2} (x-d)\big),  \no \\
&\hspace{8cm} x\in (d,b).    \no
\end{align}

In particular, 
\begin{align}
\phi(z,b,a)=\sum_{m=0}^\infty z^m\phi_m(b),\quad z\in\C,
\end{align}
where
\begin{align}
\no  \phi_0(b) &= \left[\cosh\big(V_0^{1/2}(d-c)\big)+V_0^{1/2}(c-a)\sinh\big(V_0^{1/2}(d-c)\big)\right](b-d)\\
\no  & \quad +V_0^{-1/2} \sinh\big(V_0^{1/2}(d-c)\big)+(c-a)\cosh\big(V_0^{1/2}(d-c)\big),\\[1mm]
\phi_1(b) &= \big(6V_0^{3/2}\big)^{-1} \big\{3\big[\big(aV_0(c-d)-c^2V_0+cdV_0-1\big)\sinh\big(V_0^{1/2}(c-d)\big)\\
\no & \quad +V_0^{1/2}(c-d)\cosh\big(V_0^{1/2}(c-d)\big)\big] +V_0\big[\sinh\big(V_0^{1/2}(d-c)\big)(aV_0(b-d)\\
& \quad -bcV_0+cdV_0-3)+V_0^{1/2}(3a-b-3c+d)\cosh\big(V_0^{1/2}(d-c)\big)\big]  \no \\
& \qquad \times(b-d)^2\big[V_0^{1/2}\sinh\big(2V_0^{1/2}(d-c)\big)+\cosh\big(2V_0^{1/2}(d-c)\big)\big]  \no \\
& \quad +V_0^{3/2}(a-c)^3\big\},  \no  \\
\no & \text{etc.} 
\end{align}
By construction, $\phi(z,a,a)=0$, so eigenvalues are given by solving $\phi(z,b,a)=0$, or, equivalently, by solving
\begin{align}
&\tan\big(z^{1/2} (b-d)\big)    \\
&\resizebox{\textwidth}{!}{$\quad=\dfrac{-z\cos\big(z^{1/2} (c-a)\big)\sin\big((z-V_0)^{1/2}(d-c)\big)-\sqrt{z(z-V_0)}\sin\big(z^{1/2} (c-a)\big)\cos\big((z-V_0)^{1/2}(d-c)\big)}{\sqrt{z(z-V_0)}\cos\big(z^{1/2} (c-a)\big)\cos\big((z-V_0)^{1/2}(d-c)\big)-(z-V_0)\sin\big(z^{1/2} (c-a)\big)\sin\big((z-V_0)^{1/2}(d-c)\big)}.$}  \no 
\end{align}
From \eqref{2.16}, one has 
\begin{align}
\no  F_{0,0}(z)&= \bigg[\cos\big(z^{1/2} (c-a)\big)\cos\big((z-V_0)^{1/2}(d-c)\big)\\
\no  & \quad \;\; -(z-V_0)^{1/2}\dfrac{\sin\big(z^{1/2} (c-a)\big)}{z^{1/2} }\sin\big((z-V_0)^{1/2}(d-c)\big)\bigg]\dfrac{\sin\big(z^{1/2} (b-d)\big)}{z^{1/2} }\\
&\quad +\bigg[\cos\big(z^{1/2} (c-a)\big) (z-V_0)^{-1/2} \sin\big((z-V_0)^{1/2}(d-c)\big)\\
\no  &\qquad \;\; +\ z^{-1/2} \sin\big(z^{1/2} (c-a)\big) \cos\big((z-V_0)^{1/2}(d-c)\big)\bigg]\cos\big(z^{1/2} (b-d)\big),   \no \\
& \hspace*{10.35cm} z\in\C.   \no
\end{align}
Hence, applying Corollary \ref{c4.3} with $m_0=0$ one explicitly finds the sum of the inverse of these eigenvalues, namely
\begin{align} \lb{5.65}
&\zeta(1;T_{0,0})=\text{\rm tr}_{L^2_r((a,b))}\big(T^{-1}_{0,0}\big) = - \phi_1(b)/\phi_0(b)   \\
&\scalebox{.8}{$\quad=-\big\{6V_0\big[(V_0(c-a)(b-d)+1)\sinh\big(V_0^{1/2}(d-c)\big)-V_0^{1/2}(a-b-c+d)\cosh\big(V_0^{1/2}(d-c)\big)\big]\big\}^{-1}$}   \no \\
&\scalebox{.8}{$\quad\ \ \times \big\{3\big[\big(aV_0(c-d)-c^2V_0+cdV_0-1\big)\sinh\big(V_0^{1/2}(c-d)\big)+V_0^{1/2}(c-d)\cosh\big(V_0^{1/2}(c-d)\big)\big]$}  \no \\
&\scalebox{.8}{$\quad\ \ +V_0\big[\sinh\big(V_0^{1/2}(d-c)\big)(aV_0(b-d)-bcV_0+cdV_0-3)+V_0^{1/2}(3a-b-3c+d)\cosh\big(V_0^{1/2}(d-c)\big)\big]$}  \no \\
&\scalebox{.8}{$\quad\ \ \times(b-d)^2\big[V_0^{1/2}\sinh\big(2V_0^{1/2}(d-c)\big)+\cosh\big(2V_0^{1/2}(d-c)\big)\big]+V_0^{3/2}(a-c)^3\big\}.$}  \no
\end{align}
\end{example}

Taking the limits $c\downarrow a$ and $d\uparrow b$ of \eqref{5.65} recovers the expression in Example \ref{e5.6}. Furthermore, taking the limit $V_0\downarrow0$ recovers the same expression as in Example \ref{e5.1}. The expression for $n=2$ is significantly longer and hence it is omitted here.

\subsection{Example of a Negative Constant Potential} \lb{s5.3}
\hfill

Next, we derive spectral $\zeta$-function values for the case of a negative constant potential. This case is dealt with separately since the question as to whether $z=0$ is an eigenvalue of $T_{0,0}$ depends on the actual constant value of the potential.

\begin{example}\lb{e5.9}
Let $V_0 \in (0,\infty)$, consider $q(x) = - V_0$, $x\in(a,b)$, 
and denote by $T_{0,0}$ the associated Schr\"odinger operator with Dirichlet boundary conditions at $a$ and $b$. Then,
\begin{align}\lb{5.67}
\begin{split} 
&\phi(z,x,a)= (z+V_0)^{-1/2} \sin\big((z+V_0)^{1/2}(x-a)\big),      \\ 
& \theta(z,x,a)=\cos\big((z+V_0)^{1/2}(x-a)\big),\quad z\in\C .
\end{split} 
\end{align}
Furthermore, eigenvalues and eigenfunctions for $T_{0,0}$ with $q(x)=-V_0<0,\ x\in(a,b),$ are given by
\begin{align}
    \lambda_k=\dfrac{k^2\pi^2}{(b-a)^2}-V_0,\quad y_k(x)
    = (\lambda_k+V_0)^{-1/2} \sin\big((\lambda_k+V_0)^{1/2}(x-a)\big),\quad k\in\N,
\end{align}
where one notes that due to $q(x)=-V_0<0$, $z=0$ is an eigenvalue of $T_{0,0}$ for certain values of $V_0$. Specifically, if one has
\begin{align}
V_0= k^2\pi^2/(b-a)^2, \text{ for some }k\in\N,
\end{align}
then $z=0$ is a simple eigenvalue of $T_{0,0}$. Otherwise, $z=0$ is not an eigenvalue of $T_{0,0}$. 
Moreover,
\begin{align}
F_{0,0}(z)= (z+V_0)^{-1/2} \sin\big((z+V_0)^{1/2}(b-a)\big),\quad z\in\C.
\end{align}
Applying Corollary \ref{c4.3} with $m_0=0$ when $V_0\neq k^2\pi^2/(b-a)^2$, $k\in\N$, one finds for $n=1,2,3$ $($the expression for $n=4$ is significantly longer and hence is omitted here$)$,
\begin{align}
\no  &\zeta(1;T_{0,0})=\sum_{k=1}^\infty \left[\frac{k^2\pi^2}{(b-a)^2}-V_0\right]^{-1}=\text{\rm tr}_{L^2_r((a,b))}\big(T_{0,0}^{-1}\big)\\
\no  &\hspace{1.38cm}= \big[V_0^{1/2}(a-b)\cot\big(V_0^{1/2}(b-a)\big)+1\big]\big/(2V_0),\\
\no  &\resizebox{\textwidth}{!}{$\zeta(2;T_{0,0})=\dfrac{V_0^{1/2}(b-a)\sin\big(2V_0^{1/2}(b-a)\big)+2\cos\big(2V_0^{1/2}(b-a)\big)+2V_0(b-a)^2-2}{8V_0^2\sin^2\big(V_0^{1/2}(b-a)\big)},$}\\
\no  &\zeta(3;T_{0,0})=\big(64V_0^3\sin^2\big(V_0^{1/2}(b-a)\big)\big)^{-1}\big[-12V_0(b-a)^2-16\cos\big(2V_0^{1/2}(b-a)\big)\\
\no  &\hspace{1.38cm}\quad+16-V_0^{1/2}(b-a)\big(8a^2V_0-16abV_0+8b^2V_0-3\big)\cot\big(V_0^{1/2}(b-a)\big)\\
\no  &\hspace{1.38cm}\quad-3aV_0^{1/2}\cos\big(3V_0^{1/2}(b-a)\big)\big(\sin\big(V_0^{1/2}(b-a)\big)\big)^{-1}\\
&\hspace{1.38cm}\quad+3bV_0^{1/2}\cos\big(3V_0^{1/2}(b-a)\big)\big(\sin\big(V_0^{1/2}(b-a)\big)\big)^{-1}\big].
\lb{5.71}
\end{align}

When $V_0=k_0^2\pi^2/(b-a)^2$ for some $k_0\in\N$, applying Corollary \ref{c4.3} with $m_0=1$ one finds for $n=1,2$ $($the expressions for $n=3,4$ are significantly longer and hence are omitted here$)$,
\begin{align}
\no  \zeta(1;T_{0,0})&= \sum_{\underset{k\neq k_0}{k=1}}^\infty \left[\frac{k^2\pi^2}{(b-a)^2}-V_0\right]^{-1}=\frac{\pi^2}{(b-a)^2}\sum_{\underset{k\neq k_0}{k=1}}^\infty \big[k^2-k_0^2\big]^{-1}\\
\no  &= \dfrac{\big(V_0(b-a)^2-3\big)\sin\big(V_0^{1/2}(a-b)\big)+3V_0^{1/2}(a-b)\cos\big(V_0^{1/2}(a-b)\big)}{4V_0\big(\sin\big(V_0^{1/2}(b-a)\big)+V_0^{1/2}(a-b)\cos\big(V_0^{1/2}(b-a)\big)\big)},\\
\no  \zeta(2;T_{0,0})&= \dfrac{1}{24V_0^2\big(\sin\big(V_0^{1/2}(b-a)\big)+V_0^{1/2}(a-b)\cos\big(V_0^{1/2}(b-a)\big)\big)}\\
\no  & \quad \times \big\{2\big[3\big(5-2V_0(b-a)^2\big)\sin\big(V_0^{1/2}(a-b)\big)\\
\no  &\qquad-V_0^{1/2}(b-a)(V_0(b-a)^2-15)\cos\big(V_0^{1/2}(a-b)\big)\big]\\
\no  &\qquad+3\big(\sin\big(V_0^{1/2}(b-a)\big)\big)^{-1}\big[\big(V_0(b-a)^2-3)\sin\big(V_0^{1/2}(a-b)\big)\\
\no  &\qquad-3V_0^{1/2}(b-a)\cos\big(V_0^{1/2}(a-b)\big)\big]\\
&\qquad\quad\times\big[\sin\big(V_0^{1/2}(b-a)\big)-V_0^{1/2}(b-a)\cos\big(V_0^{1/2}(b-a)\big)\big]\big\}.
\end{align}
\end{example}

Taking the limit $V_0\downarrow0$ of \eqref{5.71} recovers the expressions in Example \ref{e5.1}.

\begin{remark} \lb{r5.10} 
In the case $z=0$ is not an eigenvalue, one can verify these results via the method outlined in Remark \ref{r5.7}. Namely, letting
\begin{align}
m^2=- (b-a)^2V_0\pi^{-2}<0
\end{align}
so that
\begin{align}
m= (i/\pi)(b-a)V_0^{1/2}
\end{align}
in \eqref{5.55} and \eqref{5.58}, one verifies the expressions for $n=1,2$ as before.\hfill$\diamond$
\end{remark}

\subsection{Example of a Linear Potential} \lb{s5.4}
\hfill

We finish with an example for calculating spectral $\zeta$-function values for the linear potential, 
$q(x)=x$, $x \in (a,b)$.

\begin{example}\lb{e5.11}
Consider $q(x) = x$, $x\in(a,b)$, 
and denote by $T_{0,0}$ the associated Schr\"odinger operator with Dirichlet boundary conditions at $a$ and $b$. Then, noting that $W(\Ai,\Bi)(x)=\pi^{-1}$ $($cf. \cite[Eq. 10.4.10]{AS72}$)$, one finds
\begin{align}\lb{5.76}
\phi(z,x,a)&=\pi [\Ai(a-z)\Bi(x-z)-\Bi(a-z)\Ai(x-z)],\\ 
\theta(z,x,a)&=-\pi [\Ai'(a-z)\Bi(x-z)-\Bi'(a-z)\Ai(x-z)],
\quad z\in\C,
\end{align}
where $\Ai(\dott)$ and $\Bi(\dott)$ represent the Airy functions of the first and second kind, respectively $($cf. \cite[Sect. 10.4]{AS72}$)$. In particular, substituting $z=0$ in \eqref{5.76} yields
\begin{align}
\phi_0(x)=\pi [\Ai(a)\Bi(x)-\Bi(a)\Ai(x)],\quad \theta_0(x)=-\pi [\Ai'(a)\Bi(x)-\Bi'(a)\Ai(x)],
\end{align}
and thus the Volterra Green's function becomes
\begin{align}
g(0,x,x')=\pi [\Ai(x)\Bi(x')-\Ai(x')\Bi(x)].
\end{align}
Hence,
\begin{align}
\phi(z,b,a)=\sum_{m=0}^\infty z^m\phi_m(b),\quad z\in\C,
\end{align}
where
\begin{align}
\no & \phi_0(b) = \pi [\Ai(a)\Bi(b)-\Bi(a)\Ai(b)],    \no \\
& \phi_1(b) = \pi^{2}\int_a^b dx_1\ [\Ai(b)\Bi(x_1)-\Ai(x_1)\Bi(b)][\Ai(a)\Bi(x_1)-\Bi(a)\Ai(x_1)]    \no \\
& \hspace*{.8cm} = \pi^2\big\{\Ai(a) \Ai(b) \big[\Bi'(a)^2-\Bi'(b)^2\big] + \Bi(a) \Bi(b)\big[\Ai'(a)^2-\Ai'(b)^2\big]\\
& \hspace*{1.2cm}  + [\Ai'(b)\Bi'(b)-\Ai'(a) \Bi'(a)][\Bi(a)\Ai(b)+\Ai(a)\Bi(b)]\big\}, \no \\
& \hspace*{1.2cm} \text{etc.}    \no 
\end{align}

Furthermore, one has by construction, $\phi(z,a,a)=0$, so eigenvalues are given by solving $\phi(z,b,a)=0$, or, equivalently, by solving $\Ai(a-z)\Bi(b-z)=\Bi(a-z)\Ai(b-z)$. In particular, the characteristic function is given by
\begin{align}
F_{0,0}(z)=\pi [\Ai(a-z)\Bi(b-z)-\Bi(a-z)\Ai(b-z)],\quad z\in\C.
\end{align}
If zero is not an eigenvalue, applying Corollary \ref{c4.3} with $m_0=0$ one does find the sum of the inverse of these eigenvalues, namely
\begin{align}
\no  \zeta(1;T_{0,0})&= \tr_{L_r^2((a,b))}\big(T_{0,0}^{-1}\big) = - \phi_1(b)/\phi_0(b)     \\
\no  &= \pi [\Bi(a)\Ai(b)-\Ai(a)\Bi(b)]^{-1}\big\{\Ai(a) \Ai(b) \big[\Bi'(a)^2-\Bi'(b)^2\big]\\
&\quad +\Bi(a) \Bi(b)\big[\Ai'(a)^2-\Ai'(b)^2\big]\\
\no  &\quad +[\Ai'(b)\Bi'(b)-\Ai'(a) \Bi'(a)][\Bi(a)\Ai(b)+\Ai(a)\Bi(b)]\big\}.
\end{align}
\end{example}


\medskip

\noindent 
{\bf Acknowledgments.}  We are indebted to Angelo Mingarelli for very helpful discussions. 

 
\end{document}